\documentclass[12pt, a4paper]{amsart}
\usepackage{amsmath,amsfonts,amssymb,amsthm}
\usepackage{mathrsfs}
\usepackage[all]{xy}
 \usepackage{relsize}
\usepackage{hyperref}
\usepackage{accents}
\usepackage{color}
\usepackage{soul}
\allowdisplaybreaks
\begin{document}

\newtheorem{thm}{Theorem}[section]
\newtheorem{prop}[thm]{Proposition}
\newtheorem{coro}[thm]{Corollary}
\newtheorem{conj}[thm]{Conjecture}
\newtheorem{example}[thm]{Example}
\newtheorem{lem}[thm]{Lemma}
\newtheorem{rem}[thm]{Remark}
\newtheorem{hy}[thm]{Hypothesis}
\newtheorem*{acks}{Acknowledgements}
\theoremstyle{definition}
\newtheorem{de}[thm]{Definition}
\newtheorem{ex}[thm]{Example}
\xymatrixcolsep{5pc}

\newcommand{\C}{{\mathbb{C}}}
\newcommand{\Z}{{\mathbb{Z}}}
\newcommand{\N}{{\mathbb{N}}}
\newcommand{\Q}{{\mathbb{Q}}}
\newcommand{\te}[1]{\textnormal{{#1}}}
\newcommand{\set}[2]{{
    \left.\left\{
        {#1}
    \,\right|\,
        {#2}
    \right\}
}}
\newcommand{\sett}[2]{{
    \left\{
        {#1}
    \,\left|\,
        {#2}
    \right\}\right.
}}

\newcommand{\choice}[2]{{
\left[
\begin{array}{c}
{#1}\\{#2}
\end{array}
\right]
}}
\def \<{{\langle}}
\def \>{{\rangle}}

\def\({\left(}

\def\){\right)}

\newcommand{\overit}[2]{{
    \mathop{{#1}}\limits^{{#2}}
}}
\newcommand{\belowit}[2]{{
    \mathop{{#1}}\limits_{{#2}}
}}

\newcommand{\wt}[1]{\widetilde{#1}}

\newcommand{\wh}[1]{\widehat{#1}}

\newcommand{\no}[1]{{
    \mathopen{\overset{\circ}{
    \mathsmaller{\mathsmaller{\circ}}}
    }{#1}\mathclose{\overset{\circ}{\mathsmaller{\mathsmaller{\circ}}}}
}}

\newlength{\dhatheight}
\newcommand{\dwidehat}[1]{%
    \settoheight{\dhatheight}{\ensuremath{\widehat{#1}}}%
    \addtolength{\dhatheight}{-0.45ex}%
    \widehat{\vphantom{\rule{1pt}{\dhatheight}}%
    \smash{\widehat{#1}}}}
\newcommand{\dhat}[1]{%
    \settoheight{\dhatheight}{\ensuremath{\hat{#1}}}%
    \addtolength{\dhatheight}{-0.35ex}%
    \hat{\vphantom{\rule{1pt}{\dhatheight}}%
    \smash{\hat{#1}}}}

\newcommand{\dwh}[1]{\dwidehat{#1}}


\newcommand{\g}{{\frak g}}
\newcommand{\gc}{{\bar{\g'}}}
\newcommand{\h}{{\frak h}}
\newcommand{\cent}{{\frak c}}
\newcommand{\D}{{\mathcal D}}
\newcommand{\notc}{{\not c}}
\newcommand{\Loop}{{\cal L}}
\newcommand{\G}{{\cal G}}
\newcommand{\al}{\alpha}
\newcommand{\alck}{\al^\vee}
\newcommand{\be}{\beta}
\newcommand{\beck}{\be^\vee}
\newcommand{\ssl}{{\mathfrak{sl}}}

\newcommand{\rtu}{{\xi}}
\newcommand{\period}{{N}}
\newcommand{\half}{{\frac{1}{2}}}
\newcommand{\quar}{{\frac{1}{4}}}
\newcommand{\oct}{{\frac{1}{8}}}
\newcommand{\hex}{{\frac{1}{16}}}
\newcommand{\reciprocal}[1]{{\frac{1}{#1}}}
\newcommand{\inverse}{^{-1}}
\newcommand{\SumInZm}[2]{\sum\limits_{{#1}\in\Z_{#2}}}
\newcommand{\uce}{{\mathfrak{uce}}}


\newcommand{\orb}[1]{|\mathcal{O}({#1})|}
\newcommand{\up}{_{(p)}}
\newcommand{\uq}{_{(q)}}
\newcommand{\upq}{_{(p+q)}}
\newcommand{\uz}{_{(0)}}
\newcommand{\uk}{_{(k)}}
\newcommand{\nsum}{\SumInZm{n}{\period}}
\newcommand{\ksum}{\SumInZm{k}{\period}}
\newcommand{\overN}{\reciprocal{\period}}
\newcommand{\df}{\delta\left( \frac{\xi^{k}w}{z} \right)}
\newcommand{\dfl}{\delta\left( \frac{\xi^{\ell}w}{z} \right)}
\newcommand{\ddf}{\left(D\delta\right)\left( \frac{\xi^{k}w}{z} \right)}

\newcommand{\ldfn}[1]{{\left( \frac{1+\xi^{#1}w/z}{1-{\xi^{#1}w}/{z}} \right)}}
\newcommand{\rdfn}[1]{{\left( \frac{{\xi^{#1}w}/{z}+1}{{\xi^{#1}w}/{z}-1} \right)}}
\newcommand{\ldf}{{\ldfn{k}}}
\newcommand{\rdf}{{\rdfn{k}}}
\newcommand{\ldfl}{{\ldfn{\ell}}}
\newcommand{\rdfl}{{\rdfn{\ell}}}

\newcommand{\kprod}{{\prod\limits_{k\in\Z_N}}}
\newcommand{\lprod}{{\prod\limits_{\ell\in\Z_N}}}
\newcommand{\E}{{\mathcal{E}}}
\newcommand{\F}{{\mathcal{F}}}
\newcommand{\cS}{{\mathcal{S}}}

\newcommand{\Etopo}{{\mathcal{E}_{\te{topo}}}}

\newcommand{\Ye}{{\mathcal{Y}_\E}}

\newcommand{\rh}{{{\bf h}}}
\newcommand{\rp}{{{\bf p}}}
\newcommand{\rrho}{{{\pmb \varrho}}}
\newcommand{\ral}{{{\pmb \al}}}

\newcommand{\comp}{{\mathfrak{comp}}}
\newcommand{\ctimes}{{\widehat{\boxtimes}}}
\newcommand{\ptimes}{{\widehat{\otimes}}}
\newcommand{\ptimeslt}{{
{}_{\te{t}}\ptimes
}}
\newcommand{\ptimesrt}{{\ot_{\te{t}} }}
\newcommand{\ttp}[1]{{
    {}_{{#1}}\ptimes
}}
\newcommand{\bigptimes}{{\widehat{\bigotimes}}}
\newcommand{\bigptimeslt}{{
{}_{\te{t}}\bigptimes
}}
\newcommand{\bigptimesrt}{{\bigptimes_{\te{t}} }}
\newcommand{\bigttp}[1]{{
    {}_{{#1}}\bigptimes
}}

\newcommand{\ot}{\otimes}

\newcommand{\affva}[1]{V_{\wh\g}\(#1,0\)}
\newcommand{\saffva}[1]{L_{\wh\g}\(#1,0\)}
\newcommand{\saffmod}[1]{L_{\wh\g}\(#1\)}

\newcommand{\tar}{{\mathcal{DY}}_0\(\mathfrak{gl}_{\ell+1}\)}
\newcommand{\U}{{\mathcal{U}}}
\newcommand{\htar}{\mathcal{DY}_\hbar\(A\)}
\newcommand{\hhtar}{\widetilde{\mathcal{DY}}_\hbar\(A\)}
\newcommand{\htarz}{\mathcal{DY}_0\(\mathfrak{gl}_{\ell+1}\)}
\newcommand{\hhtarz}{\widetilde{\mathcal{DY}}_0\(A\)}
\newcommand{\qhei}{\U_\hbar\left(\hat{\h}\right)}
\newcommand{\n}{{\mathfrak{n}}}
\newcommand{\vac}{{{\bf 1}}}
\newcommand{\vtar}{{{
    \mathcal{V}_{\hbar,\tau}\left(\ell,0\right)
}}}

\newcommand{\qtar}{
    \U_q\(\wh\g_\mu\)}
\newcommand{\rk}{{\bf k}}

\newcommand{\hctvs}[1]{Hausdorff complete linear topological vector space}
\newcommand{\hcta}[1]{Hausdorff complete linear topological algebra}
\newcommand{\ons}[1]{open neighborhood system}
\newcommand{\B}{\mathcal{B}}
\newcommand{\rx}{{\bf x}}
\newcommand{\re}{{\bf e}}
\newcommand{\rphi}{{\boldsymbol{ \phi}}}

\newcommand{\der}{\mathcal D}

\newcommand{\PDer}{\te{PDer}}
\newcommand{\PEnd}{\te{PEnd}}

\makeatletter \@addtoreset{equation}{section}
\def\theequation{\thesection.\arabic{equation}}
\makeatother \makeatletter


\title{Deforming vertex algebras by vertex bialgebras}


\author{Naihuan Jing$^1$}
\address{Department of Mathematics, North Carolina State University, Raleigh, NC 27695,
USA}
\email{jing@math.ncsu.edu}
\thanks{$^1$Partially supported by NSF of China (No.11531004) and Simons Foundation (No.198129).}

\author{Fei Kong$^2$}
\address{Key Laboratory of Computing and Stochastic Mathematics (Ministry of Education), School of Mathematics and Statistics, Hunan Normal University, Changsha, China 410081} \email{kongmath@hunnu.edu.cn}
\thanks{$^2$Partially supported by NSF of China (No.11701183).}

\author{Haisheng Li}
\address{Department of Mathematical Sciences, Rutgers University, Camden, NJ 08102, USA}
\email{hli@camden.rutgers.edu}

 \author{Shaobin Tan$^4$}
 \address{School of Mathematical Sciences, Xiamen University,
 Xiamen, China 361005} \email{tans@xmu.edu.cn}
 \thanks{$^4$Partially supported by NSF of China (No.11531004).}


\begin{abstract}
This is a continuation of a previous study initiated by one of us on nonlocal vertex bialgebras and smash product nonlocal vertex algebras.
In this paper, we study a notion of right $H$-comodule nonlocal vertex algebra for a nonlocal vertex bialgebra $H$ and
give a construction of deformations of vertex algebras with a right $H$-comodule nonlocal vertex algebra structure
and a compatible $H$-module nonlocal vertex algebra structure. We also give a construction of $\phi$-coordinated quasi modules for
smash product nonlocal vertex algebras.
As an example, we give a family of quantum vertex algebras by deforming the vertex algebras
associated to non-degenerate even lattices.
\end{abstract}

\maketitle

\section{Introduction}
Vertex algebras (see \cite{bor}, \cite{FLM}) are analogues and
generalizations of commutative associative algebras,
while nonlocal vertex algebras (see \cite{bk}, \cite{li-g1}; cf. \cite{li-nonlocal})
are analogues 
of noncommutative associative algebras.
The 
(weak) quantum vertex algebras \cite{li-nonlocal}, which are certain variations of quantum vertex operator algebras
in the sense of Etingof-Kazhdan \cite{EK-qva}, are a special family of nonlocal vertex algebras.
To a certain extent, the notion of nonlocal vertex algebra plays the same role in the general vertex algebraic theory
as the notion of associative algebra does in the classical algebraic theory.
On the other hand, quantum vertex (operator) algebras, which from various viewpoints
are analogues of quantum groups (algebras), form a perfect category.

In \cite{Li-smash}, certain vertex-algebra analogues of 
Hopf algebras were introduced and studied.
Specifically, notions of nonlocal vertex bialgebra and
(left) $H$-module nonlocal vertex algebra for a nonlocal vertex bialgebra $H$ were introduced, and then
the smash product nonlocal vertex algebra $V\sharp H$ of an $H$-module nonlocal vertex algebra $V$ with $H$ was constructed.
As an application, a new construction of the lattice vertex algebras
and their modules was given.

In vertex algebra theory, among the most important notions are those of modules and $\sigma$-twisted modules
for a vertex algebra $V$ with $\sigma$ a finite order automorphism.
A notion of quasi module, generalizing that of module, was introduced in \cite{li-new}
in order to associate vertex algebras to a certain family of infinite-dimensional Lie algebras.
Indeed, with this notion vertex algebras can be associated to a wide variety of infinite-dimensional Lie algebras.
(From a certain point of view, the notion of quasi module also generalizes that of twisted module (see \cite{Li-twisted-quasi}).)
Furthermore,  in order to associate quantum vertex algebras to algebras such as quantum affine algebras,
a theory of 
$\phi$-coordinated quasi modules for nonlocal vertex algebras (including quantum vertex algebras)
was developed in \cite{li-cmp, li-jmp}, where weak quantum vertex algebras were associated to quantum affine algebras {\em conceptually}.

In the aforementioned notion of $\phi$-coordinated module, the symbol $\phi$ refers to
an associate of the $1$-dimensional additive formal group (law),
which by definition is $F(x,y):=x+y$ (an element of $\C[[x,y]]$) and which satisfies
$$F(0,x)=x=F(x,0),\quad   F(x,F(y,x))=F(F(x,y),z).$$
This very formal group law (implicitly) plays an important role  in the theory of vertex algebras and their modules.
In contrast, an associate of $F(x,y)$ is a formal series $\phi(x,z)\in \C((x))[[z]]$, satisfying the condition
$$\phi(x,0)=x,\    \   \  \phi( \phi(x,y),z)=\phi(x,y+z).$$
It was proved therein that for any $p(x)\in \C((x))$, the formal series $\phi(x,z)$ defined by
$\phi(x,z)=e^{zp(x)\frac{d}{dx}}x$ is an associate and every associate is of this form.
The essence of \cite{li-cmp} is that  a theory of $\phi$-coordinated modules
for a general nonlocal vertex algebra is attached to each associate $\phi(x,z)$,
where the usual theory of modules becomes a special case with $\phi$
taken to be the formal group law {\em itself.}
The importance lies in the fact that a wide variety of (quantum associative)  algebras
can be associated to $\phi$-coordinated (quasi) modules
for some (weak) quantum vertex algebras by choosing a suitable associate $\phi$.

What is a $\phi$-coordinated quasi module?
 Let $\phi(x,z)$ be an associate of $F(x,y)$.
Note that in the definition of modules for an associative algebra and for a (nonlocal) vertex algebra,
 the key ingredient is associativity. For a general nonlocal vertex algebra $V$,
the main defining property of a $\phi$-coordinated quasi $V$-module $W$ with vertex operator map $Y_W(\cdot,x)$
(the representation morphism) is that for any $u,v\in V$, there exists a nonzero polynomial $q(x_1,x_2)$ such that
$$q(x_1,x_2)Y_W(u,x_1)Y_W(v,x_2)\in \te{Hom}(W,W((x_1,x_2))),$$
$$\left(q(x_1,x_2)Y_W(u,x_1)Y_W(v,x_2)\right)|_{x_1=\phi(x_2,z)}=q(\phi(x_2,z),x_2)Y_W(Y(u,z)v,x_2).$$
The usual modules including the adjoint module are simply $\phi$-coordinated (quasi) modules with $\phi(x,z)=F(x,z)=x+z$.
In practice, for some algebras, their modules of highest weight type cannot be associated with vertex algebras
 in terms of usual modules, but can be viewed as $\phi$-coordinated (quasi) modules for some vertex algebras or
quantum vertex algebras with some $\phi$.

In this paper, we develop the theory of smash product nonlocal vertex algebras further in several 
directions,
with an ultimate goal to construct certain desired quantum vertex (operator) algebras.
Among the main results, we first introduce a notion
of right $H$-comodule nonlocal vertex algebra for a nonlocal vertex bialgebra $H$
and then using a compatible right $H$-comodule nonlocal vertex algebra
structure on an $H$-module nonlocal vertex algebra $V$, we construct a deformed nonlocal vertex algebra structure on $V$.
We show that under certain conditions, the deformed nonlocal vertex algebras are quantum vertex algebras.
In another direction (representation aspect),
we construct $\phi$-coordinated quasi modules for smash product nonlocal vertex algebras and
for the aforementioned deformed nonlocal vertex algebras.
We also apply the general results to the vertex algebras associated to non-degenerate even lattices,
to obtain a family of quantum vertex algebras.

Now, we give a more detailed description of the contents.
Recall that a nonlocal vertex bialgebra is simply a nonlocal vertex algebra $V$
equipped with a classical coalgebra structure on $V$ such that the comultiplication $\Delta$ and the counit $\epsilon$ are
homomorphisms of nonlocal vertex algebras. For a nonlocal vertex bialgebra $H$, an $H$-module nonlocal vertex algebra
is a nonlocal vertex algebra $V$ equipped with a module structure $Y_V^H(\cdot,x)$
for $H$ viewed as a nonlocal vertex algebra such that
\begin{eqnarray*}
&&Y_V^H(h,x)v\in V\otimes \C((x)),\\
&&Y_V^H(h,x){\bf 1}_V=\epsilon(h){\bf 1}_V,\\
&&Y_V^H(h,x)Y(u,z)v=\sum Y\left(Y_V^H(h_{(1)},x-z)u,z\right)Y_V^H(h_{(2)},x)v
\end{eqnarray*}
for $h\in H,\ u,v\in V$. (The first condition is technical, while the other two are analogues of the classical counterparts.)
Given an $H$-module nonlocal vertex algebra $V$, we have a
smash product nonlocal vertex algebra $V\sharp H$, where $V\sharp H=V\ot H$ as a vector space and
$$Y^{\sharp}(u\ot h,x)(v\ot k)=\sum Y(u,x)Y_V^H(h_{(1)},x)v\ot Y(h_{(2)},x)k$$
for $u,v\in V,\  h,k\in H$.

Let $H$ be a nonlocal vertex bialgebra. A {\em right $H$-comodule nonlocal vertex algebra} is defined to be a nonlocal vertex algebra $V$
equipped with a comodue structure $\rho: V\rightarrow V\otimes H$ for $H$ viewed as a coalgebra such that
$\rho$ is a homomorphism of nonlocal vertex algebras. Let $V$ be an $H$-module nonlocal vertex algebra.
We say a right $H$-comodule structure $\rho: V\rightarrow V\ot H$ is {\em compatible} with the (left) $H$-module structure
if $\rho$ is an $H$-module homomorphism with $H$ acting only on the first factor of $V\ot H$.
Assuming that  $\rho: V\rightarrow V\ot H$ is a compatible right $H$-comodule structure on an
$H$-module nonlocal vertex algebra $V$, we construct a new nonlocal vertex algebra ${\mathcal{D}}_{Y_V^H}^{\rho}(V)$
with $V$ as the underlying space, where the vertex operator map, denoted here by $Y_{new}(\cdot,x)$, is given by
$$Y_{new}(u,x)v=\sum Y(u_{(1)},x)Y_V^H(u_{(2)},x)v$$
for $u,v\in V$, where $\rho(u)=\sum u_{(1)}\ot u_{(2)}\in V\ot H$ in the Sweedler notation.
(Recall that $Y_V^H(\cdot,x)$ denotes the vertex operator map for $H$ on the module $V$.)
Furthermore, we show that ${\mathcal{D}}_{Y_V^H}^{\rho}(V)$  is a quantum vertex algebra,
assuming that $H$ is cocommutative, $V$ is a vertex algebra, and $Y_V^H(\cdot,x)$ is invertible
with respect to convolution (plus some technical condition).

In the theory of vertex algebras, an important family consists of
the vertex algebras $V_L$ associated to non-degenerate even lattices $L$,
which are rooted in the vertex operator realization of affine Kac-Moody algebras.
These structurally simple vertex algebras
are often used to construct or study more complicated vertex operator algebras (see \cite{FLM} for example).
Recall that  $V_L=S(\widehat{\h}^{-})\ot \C_{\varepsilon}[L]$ as a vector space,
where $\widehat{\h}^{-}=\h\ot t^{-1}\C[t^{-1}]$ with $\h=\C\ot_{\Z}L$ is an abelian Lie algebra and
$\C_{\varepsilon}[L]$ is the $\varepsilon$-twisted group algebra of $L$ with $\varepsilon$ a certain particular $2$-cocycle of $L$.
Set $B_{L}=S(\widehat{\h}^{-})\ot \C[L]$ equipped with the tensor product bialgebra structure, where $\C[L]$ is the ordinary group algebra.
Use a natural derivation on $B_L$ and the Borcherds construction to make $B_L$ a commutative vertex algebra.
In fact, $B_L$ is a vertex bialgebra, which was exploited in \cite{Li-smash}.

In this paper, we explore the vertex bialgebra $B_L$ furthermore. We prove that
there exists a natural right $B_L$-comodule vertex algebra structure $\rho: V_L\rightarrow V_L\ot B_L$
on the vertex algebra $V_L$. In fact, with $V_L$ identified with $B_L$ as a vector space,  $\rho$ is simply the comultiplication
$\Delta: B_L\ot B_L\rightarrow B_L$. (It was proved in \cite{Li-smash} that with canonical identifications,
$\Delta$ is a vertex algebra embedding of $V_L$ into $B_{L,\varepsilon}\sharp B_L$.)
On the other hand, for any linear map $\eta(\cdot,x):\h\rightarrow \h\otimes x\C[[x]]$,
we give a compatible $B_L$-module vertex algebra structure $Y^{\eta}_M(\cdot,x)$
on $V_L$. Consequently,  we obtain a family of quantum vertex algebras $V_L^{\eta}$ (with $V_L$ as the underlying vector space).

In a sequel, we shall use $\hbar$-adic versions of the results obtained in this paper
to construct certain quantum vertex operator algebras (over $\C[[\hbar]]$)
in the sense of Etingof-Kazhdan by deforming vertex algebras $V_{L}$, where
$\phi$-coordinated quasi modules for these quantum vertex operator algebras
are associated to highest weight modules for twisted  quantum affine algebras.

This paper is organized as follows: In Section 2, we first recall the basic notions and results about
smash product  nonlocal vertex algebras, and then study right comodule nonlocal vertex algebras, and
compatible module nonlocal vertex algebra structure.
Furthermore, we give a deformation construction of nonlocal vertex algebras.
In Section 3, we continue to show that under certain conditions, the deformed nonlocal vertex algebras are quantum vertex algebras.
In Section 4, we study $\phi$-coordinated quasi modules for smash product nonlocal vertex algebras.
In Section 5, for a general  non-degenerate even lattice $L$, we give a right $H$-comodule vertex algebra structure on
$V_L$ and give a family of compatible $H$-module vertex algebra structures.
Then we give a family of deformations of the vertex algebra $V_L$.

In this paper, we work on the field $\C$ of complex numbers, and we use
$\Z_+$ for the set of positive integers, while we use $\N$ for the set of nonnegative integers.
For a ring $R$, e.g., $\Z,\  \C, \  \C[x,y]$, we use $R^{\times}$ for the set of nonzero elements of $R$.
We continue using the formal variable notations and conventions (including formal delta functions)
as established in \cite{FLM} and \cite{fhl}.

\newpage

\section{Smash product nonlocal vertex algebras}

In this section, we first recall  the basic notions and results on smash product nonlocal vertex algebras from \cite{Li-smash},
and we then study a notion of right $H$-comodule nonlocal vertex algebra for a nonlocal vertex bialgebra $H$.
As the main result, for an $H$-module nonlocal vertex algebra $V$ with a compatible right $H$-comodule nonlocal vertex algebra
structure, we establish a deformed nonlocal vertex algebra structure on $V$.

\subsection{Nonlocal vertex algebras and quantum vertex algebras}
Here, we recall some basic notions and results on nonlocal vertex algebras and
(weak) quantum vertex algebras from \cite{li-nonlocal}.

We start with the notion of nonlocal vertex algebra (see Remark \ref{nlva-definitions}).

\begin{de}\label{def-nlva}
A {\em nonlocal vertex algebra} is a vector space $V$ equipped with a linear map
\begin{eqnarray*}
Y(\cdot,x): &&  V\rightarrow (\te{End} V)[[x,x^{-1}]]\nonumber\\
&&v\mapsto Y(v,x)=\sum_{n\in \Z}v_nx^{-n-1}\  \  (\mbox{where }v_n\in \te{End} (V))
\end{eqnarray*}
and a distinguished vector ${\bf 1}\in V$, satisfying the conditions that
\begin{align}
Y(u,x)v\in V((x))\   \   \   \mbox{ for }u,v\in V,
\end{align}
\begin{align}
Y({\bf 1},x)v=v,\   \   \   \   Y(v,x){\bf 1}\in V[[x]]\   \mbox{ and }\  \lim_{x\rightarrow 0}Y(v,x){\bf 1}=v\   \   \   \mbox{ for }v\in V,
\end{align}
and that for any $u,v\in V$, there exists a nonnegative integer $k$ such that
\begin{align}\label{nlva-compatibilty}
(x_1-x_2)^{k}Y(u,x_1)Y(v,x_2)\in \te{Hom}(V,V((x_1,x_2)))
\end{align}
and
\begin{align}\label{nlva-associativity}
\left((x_1-x_2)^{k}Y(u,x_1)Y(v,x_2)\right)|_{x_1=x_2+x_0}=
x_0^{k}Y(Y(u,x_0)v,x_2).
\end{align}
\end{de}

The following was obtained in \cite{JKLT1} (cf. \cite{DLMi}, \cite{ltw}):

\begin{lem}\label{two-definitions}
Let $V$ be a nonlocal vertex algebra. Then the following
 {\em weak associativity} holds: For any $u,v,w\in V$, there exists $l\in \N$ such that
 \begin{align}\label{weak-assoc}
 (x_0+x_2)^{l}Y(u,x_0+x_2)Y(v,x_2)w= (x_0+x_2)^{l}Y(Y(u,x_0)v,x_2)w.
\end{align}
\end{lem}

\begin{rem}\label{nlva-definitions}
{\em Note that a notion of nonlocal vertex algebra
was defined in  \cite{li-nonlocal} in terms of weak associativity as in Lemma \ref{two-definitions},
which is the same as the notion of axiomatic $G_1$-vertex algebra introduced in \cite{li-g1} and
also the same as the notion of field algebra introduced in \cite{bk}.
In view of Lemma \ref{two-definitions}, the notion of nonlocal vertex algebra in the sense of
Definition \ref{def-nlva} is stronger than the other notions. }
\end{rem}

Let $V$ be a nonlocal vertex algebra. Define a linear operator $\der$ by $\der (v)= v_{-2}\vac$ for $v\in V$.
Then
\begin{align}
[\der, Y(v,x)]=Y(\der (v),x)=\frac{d}{dx}Y(v,x)\   \   \te{for }v\in V.
\end{align}

The following notion singles out a family of nonlocal vertex algebras (see \cite{li-nonlocal}):

\begin{de}\label{weak-qva}
A {\em weak quantum vertex algebra} is a nonlocal vertex algebra $V$
satisfying the following {\em ${\mathcal{S}}$-locality:}  For any $u,v\in V$, there exist
$$u^{(i)},v^{(i)}\in V,\ f_i(x)\in \C((x)) \  \   (i=1,\dots,r)$$
and a nonnegative integer $k$ such that
\begin{eqnarray}
(x-z)^{k}Y(u,x)Y(v,z)=(x-z)^{k}\sum_{i=1}^{r}f_{i}(z-x)Y(v^{(i)},z)Y(u^{(i)},x).
\end{eqnarray}
\end{de}

The following was proved in \cite{li-nonlocal}:

\begin{prop}\label{S-Jacobi-def}
A weak quantum vertex algebra can be defined equivalently by replacing the conditions (\ref{nlva-compatibilty})
and (\ref{nlva-associativity}) in Definition \ref{def-nlva} (for  a nonlocal vertex algebra) with the condition
that for any $u,v\in V$, there exist
$$u^{(i)},v^{(i)}\in V,\ f_i(x)\in \C((x)) \  \   (i=1,\dots,r)$$
 such that
 \begin{eqnarray}\label{S-Jacobi-identity-1}
 &&x_0^{-1}\delta\left(\frac{x_1-x_2}{x_0}\right)Y(u,x_1)Y(v,x_2)\nonumber\\
 &&\  \   \   \   -x_0^{-1}\delta\left(\frac{x_2-x_1}{-x_0}\right)\sum_{i=1}^{r}f_i(x_2-x_1)Y(v^{(i)},x_2)Y(u^{(i)},x_1)\nonumber\\
 &=&x_2^{-1}\delta\left(\frac{x_1-x_0}{x_2}\right)Y(Y(u,x_0)v,x_2)
 \end{eqnarray}
 (the {\em ${\mathcal{S}}$-Jacobi identity}).
\end{prop}

 The following is a convenient technical result (cf.  \cite[Proposition 5.6.7]{LL}, \cite{li-local}):

\begin{lem}\label{nonlocalva-fact}
Let $V$ be a nonlocal vertex algebra and let
$$u,v,\ u^{(i)},v^{(i)}, w^{(j)}\in V,\ f_i(x)\in \C((x))\  \  (1\le i\le r,\ 0\le j\le s)$$
such that
\begin{align*}
  (x_1-x_2)^kY(u,x_1)Y(v,x_2)=(x_1-x_2)^k\sum_{i=1}^rf_i(x_2-x_1)Y(v^{(i)},x_2)Y(u^{(i)},x_1)
\end{align*}
for some positive integer $k$.
Then
\begin{align}\label{commutator-relation-unknown}
&(x_1-x_2)^nY(u,x_1)Y(v,x_2)-(-x_2+x_1)^n\sum_{i=1}^rf_i(x_2-x_1)Y(v^{(i)},x_2)Y(u^{(i)},x_1)\nonumber\\
=\ &\sum_{j=0}^sY(w^{(j)},x_2)\frac{1}{j!}\left(\frac{\partial}{\partial x_2}\right)^j x_1^{-1}\delta\left(\frac{x_2}{x_1}\right).
\end{align}
if and only if $u_{n+j}v=w^{(j)}$ for $0\le j\le s$ and $u_{j+n}v=0$ for $j >s$.
\end{lem}

\begin{proof}
From \cite{li-nonlocal} the ${\mathcal{S}}$-Jacobi identity (\ref{S-Jacobi-identity-1}) holds, we have that
\begin{align*}
&(x_1-x_2)^nY(u,x_1)Y(v,x_2)-(-x_2+x_1)^n\sum_{i=1}^rf_i(x_2-x_1)Y(v^{(i)},x_2)Y(u^{(i)},x_1)\nonumber\\
=\ &\sum_{j\ge 0}Y(u_{n+j}v,x_2)\frac{1}{j!}\left(\frac{\partial}{\partial x_2}\right)^j x_1^{-1}\delta\left(\frac{x_2}{x_1}\right).
\end{align*}
By using this, we complete the proof.
\end{proof}

Recall that a {\em rational quantum Yang-Baxter operator} on a vector space $U$ is a linear map
 $$\cS(x):  U\ot U\rightarrow U\ot U\ot \C((x))$$
 such that
 \begin{eqnarray}
   {\mathcal{S}}^{12}(x)\cS^{13}(x+z)\cS^{23}(z)=\cS^{23}(z)\cS^{13}(x+z)\cS^{12}(x)
 \end{eqnarray}
 (the {\em quantum Yang-Baxter equation}).
Furthermore, a rational quantum Yang-Baxter operator $\cS(x)$ on $U$ is said to be {\em unitary} if
\begin{align}
  \cS^{21}(x)\cS^{12}(-x)=1,
\end{align}
where $\cS(x)=\sigma \cS(x)\sigma$ with $\sigma$ denoting the flip operator on $U\ot U$.

For any nonlocal vertex algebra $V$, follow \cite{EK-qva} to denote by $Y(x)$ the linear map
$$Y(x): V\ot V\rightarrow V((x)),$$
associated to the vertex operator map $Y(\cdot,x): V\rightarrow \te{Hom}(V,V((x)))$ of $V$.
The following is a variation of Etingof-Kazhdan's notion of quantum vertex operator algebra (see \cite{EK-qva}, \cite{li-nonlocal}):

\begin{de}\label{de:qva}
 A {\em quantum vertex algebra} is a weak quantum vertex algebra $V$ equipped with
a unitary rational quantum Yang-Baxter operator $\cS(x)$ on $V$, satisfying the conditions that
for any $u,v\in V$, there exists $k\in \N$ such that
\begin{align*}
  (x-z)^kY(x)\(1\otimes Y(z)\)\(\cS(x-z)(u\otimes v)\otimes w\)
=(x-z)^kY(z)\(1\otimes Y(x)\)(v\otimes u\otimes w)
\end{align*}
for all $w\in V$ and that
\begin{align}
  \left[\mathcal D\otimes 1,{\mathcal{S}}(x)\right]&=-\frac{d}{dx}{\mathcal{S}}(x),\\
  \cS(x)\(Y(z)\otimes 1\)&=\(Y(z)\otimes 1\)\cS^{23}(x)\cS^{13}(x+z)
\end{align}
(the {\em hexagon identity}).
\end{de}

A nonlocal vertex algebra $V$ is said to be {\em nondegenerate} (see \cite{EK-qva})
  if for every positive integer $n$, the linear map
 $$Z_n: \C((x_1))\cdots ((x_n))\otimes V^{\otimes n}\rightarrow V((x_1))\cdots ((x_n))),$$
 defined by
 $$Z_n(f\ot v_1\otimes \cdots \otimes v_n)=f Y(v_1,x_1)\cdots Y(v_n,x_n){\bf 1},$$
 is injective.

 We have (see \cite{EK-qva}, \cite{li-nonlocal}):

 \begin{prop}\label{nondeg-wqva}
 Let $V$ be a nondegenerate weak quantum vertex algebra. Then
 there exists a linear map $\cS(x): V\ot V\rightarrow V\ot V\ot \C((x))$
 such that $V$ together with $\cS(x)$ is a quantum vertex algebra. Furthermore, such a linear map $\cS(x)$
 is uniquely determined by the $\cS$-locality.
 \end{prop}

The following notion (under a different name) was introduced in \cite{li-new}:

\begin{de}\label{def-chi-G-va}
Let $G$ be a group, $\chi: G\rightarrow \C^{\times}$ a linear character.
A {\em $(G,\chi)$-module nonlocal vertex algebra} is a nonlocal vertex algebra $V$ equipped with
a representation $R:G\rightarrow \te{GL} (V)$ of $G$ on $V$  such that
$R(g)\vac=\vac$ and
\begin{align}
  R(g)Y(v,x)R(g)\inverse =Y(R(g)v,\chi(g)x)\quad\te{for }g\in G,\  v\in V.
\end{align}
We sometimes denote a $(G,\chi)$-module nonlocal vertex algebra by a pair $(V,R)$.
\end{de}

\begin{rem}
{\em Note that in Definition \ref{def-chi-G-va}, for $g\in G$, if
$\chi(g)=1$, then $g$ acts on $V$ as an automorphism.
Thus a $(G,\chi)$-module nonlocal vertex algebra with $\chi=1$ (the trivial character) is simply
a nonlocal vertex algebra on which  $G$ acts as an automorphism group.
In this case, $V$ is a $\C[G]$-module nonlocal vertex algebra in the sense of \cite{Li-smash} (see Section  3).}
\end{rem}

Let $(U,R_U)$ and $(V,R_V)$ be $(G,\chi)$-module nonlocal vertex algebras.
A {\em $(G,\chi)$-module nonlocal vertex algebra homomorphism} is a nonlocal vertex algebra homomorphism
$f:U\rightarrow V$ which is also a $G$-module homomorphism.

The following is a technical lemma formulated in \cite[Lemma 3.3]{JKLT1}:

\begin{lem}\label{lem:G-va-generating}
Let $G$ be a group equipped with a linear character $\chi: G\rightarrow \C^{\times}$.
Suppose that $V$ is a nonlocal vertex algebra,
$\rho:G\rightarrow \te{Aut}(V)$ and $L:G\rightarrow \te{GL}(V)$
are  group homomorphisms such that $L(g){\bf 1}={\bf 1}$,
\begin{align*}
&\rho(g)L(h)=L(h)\rho(g)  \quad\te{for }g,h\in G,\\
&L(g)Y(v,x)L(g)\inverse=Y(L(g)v,\chi(g)x)  \quad\te{for }g\in G,\   v\in S,
\end{align*}
where $S$ is a generating subset of $V$.
Then $(V,R)$ is a $(G,\chi)$-module nonlocal vertex algebra with $R$ defined by
$R(g)= \rho(g)L(g)$ for $g\in G$.
\end{lem}

\subsection{Smash product nonlocal vertex algebras}
We first recall from \cite{Li-smash} the basic notions and results on
smash product nonlocal vertex algebras, and then we introduce a notion of right $H$-comodule nonlocal vertex algebra
with $H$ a nonlocal vertex bialgebra and we establish a deformed nonlocal vertex algebra structure on a right $H$-comodule nonlocal vertex algebra $V$ with a compatible (left) $H$-module nonlocal vertex algebra structure.

We begin with the notion of nonlocal vertex bialgebra.

\begin{de}
A {\em nonlocal vertex bialgebra} is a nonlocal vertex algebra $V$
equipped with a classical coalgebra structure ($\Delta,\varepsilon$)
such that (the co-multiplication) $\Delta:V\rightarrow V\ot V$
and (the co-unit) $\varepsilon:V\rightarrow \C$ are homomorphisms of nonlocal vertex algebras.
\end{de}

The notion of homomorphism of nonlocal vertex bialgebras  is defined in the obvious way:
For nonlocal vertex bialgebras ($V,\Delta,\varepsilon$) and $(V',\Delta',\varepsilon')$,
{\em a nonlocal vertex bialgebra homomorphism}  from $V$ to $V'$ is  a homomorphism $f$ of
nonlocal vertex algebras such that
\begin{align}
  \Delta'\circ f=(f\ot f)\circ \Delta,\   \   \quad
  \varepsilon'\circ f=\varepsilon.
\end{align}
In other words, a nonlocal vertex bialgebra homomorphism is both
a nonlocal vertex algebra homomorphism and a coalgebra homomorphism.

\begin{de}\label{de:mod-va-for-vertex-bialg}
Let $(H,\Delta,\varepsilon)$ be a nonlocal vertex bialgebra. A (left) {\em $H$-module nonlocal vertex algebra}
is a nonlocal vertex algebra $V$ equipped with a module structure $Y_{V}^H(\cdot,x)$
on $V$ for $H$ viewed as a nonlocal vertex algebra such that
\begin{align}
  &Y_{V}^H(h,x)v\in V\ot \C((x)),\label{eq:mod-va-for-vertex-bialg1}\\
  &Y_{V}^H(h,x)\vac_V=\varepsilon(h)\vac_V,\label{eq:mod-va-for-vertex-bialg2}\\
  &Y_{V}^H(h,x_1)Y(u,x_2)v=\sum Y(Y_{V}^H(h_{(1)},x_1-x_2)u,x_2)Y_{V}^H(h_{(2)},x_1)v
  \label{eq:mod-va-for-vertex-bialg3}
\end{align}
for $h\in H$, $u,v\in V$, where $\vac_V$ denotes the vacuum vector of $V$
and $\Delta(h)=\sum h_{(1)}\ot h_{(2)}$ is the coproduct in the Sweedler notation.
\end{de}

The following two results were obtained in \cite{Li-smash}:

\begin{thm}\label{smash-product}
Let $H$ be a nonlocal vertex bialgebra and let $V$ be an $H$-module nonlocal vertex algebra.
Set $V\sharp H=V\ot H$ as a vector space.   For $u,v\in V$, $h,h'\in H$, define
\begin{align}
  Y^\sharp (u\ot h,x)(v\ot h')=\sum Y(u,x)Y(h_{(1)},x)v\ot Y(h_{(2)},x)h'.
\end{align}
Then $(V\sharp H, Y^\sharp,{\bf 1}\otimes {\bf 1})$ carries the structure of a nonlocal vertex algebra,
which contains $V$ and $H$ canonically as subalgebras
such that   for $h\in H$, $u\in V$,
\begin{align}
  Y^\sharp (h,x_1)Y^\sharp(u,x_2)=\sum Y^\sharp(Y(h_{(1)},x_1-x_2)u,x_2)Y^\sharp (h_{(2)},x_1).
\end{align}
\end{thm}

\begin{prop}\label{prop:smash-mod-1}
Let $H$ be a nonlocal vertex bialgebra and let $V$ be an $H$-module nonlocal vertex algebra.
Let $W$ be a vector space and assume that $(W,Y_W^V)$ is a $V$-module and $(W,Y_W^H)$ is an $H$-module
such that
for any $h\in H$, $v\in V$, $w\in W$,
\begin{align}
&Y_W^H(h,x)w\in W \ot \C((x)),\\
  &Y_W^H(h,x_1)Y_W^V(v,x_2)w=\sum Y_W^V(Y_V^H(h_{(1)},x_1-x_2)v,x_2)Y_W^H(h_{(2)},x_1)w.
\end{align}
 Then $W$ is a $V\sharp H$-module with the vertex operator map
$Y_W^{\sharp}(\cdot,x)$ given by
\begin{align}
  Y_W^{\sharp}(v\ot h,x)w=Y_W^V(v,x)Y_W^H(h,x)w\quad \te{ for }h\in H, \  v\in V, \ w\in W.
\end{align}
\end{prop}

As an immediate consequence of Proposition \ref{prop:smash-mod-1}, we have:

\begin{prop}\label{lem:V-BV-mod-struct-on-V}
Let $H$ be a nonlocal vertex bialgebra and let
$V$ be an $H$-module nonlocal vertex algebra.
Then $V$ is a $V\sharp H$-module with the vertex operator map $Y_V^{\sharp}(\cdot,x)$ given by
\begin{align*}
  Y_V^{\sharp}(u\ot h,x)v=Y(u,x)Y_V^H(h,x)v\  \   \  \   \te{for }u,v\in V,\  h\in H.
\end{align*}
Furthermore, $1\ot \varepsilon:V\sharp H\rightarrow V$
is a $V\sharp H$-module epimorphism.
\end{prop}

We shall need the following simple result:

\begin{lem}\label{lem:strong-asso}
Let $H$ be a nonlocal vertex bialgebra and let $(V,Y_V^H)$ be an $H$-module nonlocal vertex algebra.
 Then
\begin{align}\label{estrict-assoc-1}
  Y_V^H(Y(h,z)h',x)=Y_V^H(h,x+z)Y_V^H(h',x)\quad\te{for }h,h'\in H.
\end{align}
\end{lem}

\begin{proof} Assume $h,h'\in H$.  Let $v\in V$.  There exists a nonnegative integer $l$ such that
$$(z+x)^{l}Y_{V}^H(h,z+x)Y_{V}^H(h',x)v= (z+x)^{l}Y_{V}^H(Y(h,z)h',x)v.$$
From the assumption (\ref{eq:mod-va-for-vertex-bialg1}),  we have
$$Y_{V}^H(h',x)v,\    \  Y_{V}^H(h,x)u\in V\otimes \C((x))$$
for any $u\in V$. Consequently, we have
$$Y_{V}^{H}(h,x_1)Y_{V}^H(h',x)v\in V\otimes \C((x,x_1)).$$
This implies that $Y_V^H(h,x+z)Y_V^H(h',x)$ exists in $(\te{Hom}(V,V\otimes \C((x))))[[z]]$.
It also implies that  we can choose $l$ so large that we also have
$$x_1^lY_{V}^{H}(h,x_1)Y_{V}^H(h',x)v\in V\otimes \C((x))[[x_1]].$$
Then
$$(z+x)^{l}Y_{V}^H(h,z+x)Y_{V}^H(h',x)v= (x+z)^{l}Y_{V}^H(h,x+z)Y_{V}^H(h',x)v.$$
Thus
$$(x+z)^{l}Y_{V}^H(h,x+z)Y_{V}^H(h',x)v=(z+x)^{l}Y_{V}^H(Y(h,z)h',x)v.$$
Consequently, by cancellation we get
$$Y_{V}^H(h,x+z)Y_{V}^H(h',x)v=Y_{V}^H(Y(h,z)h',x)v,$$
as desired.
\end{proof}

As in \cite{Li-smash}, by a {\em differential bialgebra}
we mean a bialgebra $(B, \Delta,\varepsilon)$ equipped with a derivation $\partial$ such that $\varepsilon\circ \partial=0$
and $\Delta\partial=(\partial\ot 1+1\ot\partial)\Delta$. (That is, $\varepsilon$ and $\Delta$ are homomorphisms of differential algebras.)

\begin{rem}
{\em Let $(B, \Delta,\varepsilon,\partial)$ be a differential bialgebra.
In particular, the associative algebra $B$ with derivation $\partial$ is a differential algebra.
Then we have a nonlocal vertex algebra structure on $B$ with
$$Y(a,x)b=(e^{x\partial} a)b\   \   \   \mbox{ for }a,b\in B.$$
Denote this nonlocal vertex algebra by $(B,\partial)$.
Then  $(B,\partial)$ equipped with $\Delta$ and $\varepsilon$ is
naturally a nonlocal vertex bialgebra (see \cite[Example 4.2]{Li-smash}). }
\end{rem}

Let $V$ be a nonlocal vertex algebra. A subset $U$ of $\te{Hom}(V,V\ot\C((x)))$
is said to be {\em $\Delta$-closed} (see  \cite{Li-smash}) if for any $a(x)\in U$,
there exist $a_{(1i)}(x),a_{(2i)}(x)\in U$ for $i=1,2,\dots,r$ such that
\begin{align}
  a(x_1)Y(v,x_2)=\sum_{i=1}^rY(a_{(1i)}(x_1-x_2)v,x_2)a_{(2i)}(x_1)\  \  \   \
  \te{for all }v\in V.
\end{align}
Let $B(V)$ be the sum of all $\Delta$-closed subspaces $U$ of
$\te{Hom}(V,V\ot\C((x)))$ such that
\begin{align}
a(x)\vac\in \C\vac\quad\te{for }a(x)\in U.
\end{align}
Note that $\te{Hom}(V,V\ot\C((x)))\cong\te{End}_{\C((x))}(V\ot \C((x)))$
is an associative algebra.

The following is a summary of some results of \cite{Li-smash}:

\begin{prop}\label{prop:B-V-vertex-bialg}
Let $V$ be a  nonlocal vertex algebra.  Then $B(V)$ is a $\Delta$-closed subalgebra of $\te{Hom}(V,V\ot\C((x)))$
and is closed under the derivation $\partial=\frac{d}{dx}$.
Furthermore, if $V$ is nondegenerate, then
\te{(a)} $B(V)$ is a differential bialgebra with the coproduct $\Delta$
and the counit $\varepsilon$, which are uniquely determined by
\begin{align}
  &a(x)\vac=\varepsilon(a(x))\vac,\quad
  \Delta(a(x))=\sum a_{(1)}(x)\ot a_{(2)}(x)
\end{align}
for $a(x)\in B(V)$, where $a(x)Y(v,z)=\sum Y(a_{(1)}(x-z)v,z)a_{(2)}(x)$
for all $v\in V$.

\te{(b)} $V$ is a $B(V)$-module nonlocal vertex algebra with $Y_V(a(x),x_0)=a(x_0)$ for $a(x)\in B(V)$.

\te{(c)} For any nonlocal vertex bialgebra $H$, an $H$-module nonlocal vertex algebra
structure $Y_V^H(\cdot,x)$ on $V$ amounts to a nonlocal vertex bialgebra homomorphism
from $H$ to $B(V)$.
\end{prop}

The following notions are due to  \cite{EK-qva} and \cite{Li-pseudo} (cf. \cite{Li-smash}):

\begin{de}
Let $V$ be a nonlocal vertex algebra.  A {\em pseudo-derivation} (resp. {\em pseudo-endomorphism})
of $V$ is an element $a(x)\in \te{Hom}(V,V\ot \C((x)))$ such that
\begin{eqnarray}
[a(x),Y(u,z)]=Y(a(x-z)u,z)
\end{eqnarray}
(resp. $a(x)Y(u,z)=Y(a(x-z)u,z)a(x)$) for all $u\in V$.
Denote by $\PDer(V)$ (resp. $\PEnd(V)$) the set of all pseudo-derivations (resp. pseudo-endomorphisms).
\end{de}

From definition, we have
\begin{align}
\PDer(V)\subset B(V),\   \   \   \   \PEnd(V)\subset B(V).
\end{align}

The following construction of pseudo-derivations is due to \cite{EK-qva}:

\begin{prop}\label{prop:pseudo-der}
Let $V$ be a vertex algebra. For $v\in V$, $f(x)\in\C((x))$, set
\begin{align}
  \Phi(v,f)(x)=\te{Res}_{z}f(x-z)Y(v,z)=\sum_{n\ge0} \frac{(-1)^n}{n!} f^{(n)}(x)v_n,
\end{align}
where $f^{(n)}(x)=\(\frac{d}{dx}\)^nf(x)$.
Then $\Phi(v,f)(x)\in \PDer(V)$.
\end{prop}

The following is an analogue of the notion of right comodule algebra:

\begin{de}
Let $H$ be a nonlocal vertex bialgebra.
A {\em right $H$-comodule nonlocal vertex algebra} is
a nonlocal vertex algebra $V$ equipped with a homomorphism
 $\rho: V\rightarrow V\ot H$ of nonlocal vertex algebras  such that
\begin{align}\label{eq:comod-cond}
(\rho\ot 1)\rho=(1\ot\Delta)\rho,\quad
  (1\ot \epsilon)\rho=\te{Id}_V,
\end{align}
i.e., $\rho$ is also a right comodule structure on $V$ for $H$ viewed as a coalgebra.
\end{de}

\begin{rem}
{\em Let $B$ be a coalgebra and let $U$ with the map $\rho: U\rightarrow U\ot B$ be a right $B$-comodule.
For $u\in U$, we have  $(\rho\ot 1)\rho(u)=(1\ot\Delta)\rho(u)$, i.e.,
\begin{align*}
\sum u_{(1),(1)}\ot u_{(1),(2)}\ot u_{(2)}=\sum u_{(1)}\ot u_{(2),(1)}\ot u_{(2),(2)}.
\end{align*}
Note that by applying permutation operators on both sides we get five more relations. If $B$ is cocommutative,
we have $P_{23}(1\ot \Delta)\rho=(1\ot \Delta)\rho$, so
\begin{align}\label{eP23-rho-rho}
P_{23}(\rho\ot 1)\rho=(1\ot \Delta)\rho,\quad  P_{23}(\rho\ot 1)\rho=(\rho\ot 1)\rho,
\end{align}
where $P_{23}$ denotes the indicated flip operator on $B\ot B\ot B$.}
\end{rem}

\begin{de}\label{compatible-ma-cma}
Let $H$ be a nonlocal vertex bialgebra.
Assume that $V$ is an $H$-module nonlocal vertex algebra with the module vertex operator map
$$Y_V^H(\cdot,x):\  H\rightarrow \te{Hom}(V,V\ot\C((x)))$$
and $V$ is also a right $H$-comodule nonlocal vertex algebra with the comodule map
$\rho:V\rightarrow V\ot H$.
We say $Y_V^H$ and $\rho$ are {\em compatible} if
$\rho$ is an $H$-module homomorphism with $V\otimes H$
viewed as an $H$-module on which $H$ acts on the first factor only, i.e.,
\begin{align}\label{compatible-relation}
  \rho(Y_V^H(h,x)v)=(Y_V^H(h,x)\ot 1)\rho(v)\quad\te{for }h\in H,\  v\in V.
\end{align}
\end{de}

The following is the main result of this section:

\begin{thm}\label{prop:deform-va}
Let $H$ be a cocommutative nonlocal vertex bialgebra and let $V$ be a nonlocal vertex algebra.
Suppose that $(V,Y_V^H)$ is an $H$-module nonlocal vertex algebra and
$(V,\rho)$ is a right $H$-comodule nonlocal vertex algebra such that $Y_V^H$ and $\rho$ are compatible.
For $a\in V$, set
\begin{align}
\mathfrak D_{Y_V^H}^\rho (Y)(a,x)=\sum Y(a_{(1)},x)Y_V^H(a_{(2)},x)
\end{align}
on $V$, where $\rho(a)=\sum a_{(1)}\ot a_{(2)}\in V\otimes H$.
Then $(V,\mathfrak D_{Y_V^H}^\rho (Y),\vac)$ carries the structure of a nonlocal vertex algebra.
Denote this nonlocal vertex algebra by $\mathfrak D_{Y_V^H}^\rho (V)$.
Furthermore, $\rho$ is a nonlocal vertex algebra homomorphism from
$\mathfrak D_{Y_V^H}^\rho (V)$ to $V\sharp H$.
\end{thm}

\begin{proof} Recall  from Proposition \ref{lem:V-BV-mod-struct-on-V} that $Y_V^{\sharp}(\cdot,x)$
denotes the $V\sharp H$-module structure on $V$.
Notice that for $a\in V$, we have
\begin{align}\label{2.35}
\mathfrak D_{Y_V^H}^\rho (Y)(v,x)=Y_V^{\sharp}(\rho(v),x).
\end{align}
It follows that $\mathfrak D_{Y_V^H}^\rho (Y)(a,x)\in \te{Hom}(V,V((x)))$.
Let $u,v\in V$. As $H$ is cocommutative, by (\ref{eP23-rho-rho}) we have
\begin{align*}
  \sum u_{(1),(1)}\ot u_{(2)}\ot u_{(1),(2)}
  =\sum u_{(1)}\ot u_{(2),(1)}\ot u_{(2),(2)}.
\end{align*}
Using this and (\ref{compatible-relation}),  we get
\begin{align}\label{rho-homomorphism}
  \rho\left(\mathfrak D_{Y_V^H}^\rho (Y)(u,x)v\right)
  =&\sum\rho\(Y(u_{(1)},x)Y_V^H(u_{(2)},x)v\)\nonumber\\
  =\  &\sum Y^\ot (\rho(u_{(1)}),x)\rho\(Y_V^H(u_{(2)},x)v\)\nonumber\\
 =\  &\sum Y^\ot (\rho(u_{(1)}),x)(Y_V^H(u_{(2)},x)\ot 1)\rho(v)\nonumber\\
 =\  &\sum Y^\ot (u_{(1),(1)}\ot u_{(1),(2)},x)(Y_V^H(u_{(2)},x)\ot 1)(v_{(1)}\ot v_{(2)})\nonumber\\
 =\  &\sum Y(u_{(1),(1)},x)Y_V^H(u_{(2)},x)v_{(1)}\ot Y(u_{(1),(2)},x)v_{(2)}\nonumber\\
  =\  &\sum Y(u_{(1)},x)Y_V^H(u_{(2),(1)},x)v_{(1)}\ot Y(u_{(2),(2)},x)v_{(2)}\nonumber\\
  =\  &\sum Y^\sharp(u_{(1)}\ot u_{(2)},x)(v_{(1)}\ot v_{(2)})\nonumber\\
  =\  &Y^\sharp(\rho(u),x)\rho(v).
\end{align}
Now, let $u,v\in V$. Then there exists a nonnegative integer $k$ such that
\begin{align*}
&(x_1-x_2)^k Y_V^{\sharp}(\rho(u),x_1)Y_V^{\sharp}(\rho(v),x_2)\in \te{Hom}(V,V((x_1,x_2))),\\
 &   \left((x_1-x_2)^k Y_V^{\sharp}(\rho(u),x_1)Y_V^{\sharp}(\rho(v),x_2)\right)|_{x_1=x_2+x_0}
 =x_0^k Y_V^{\sharp}(Y^{\sharp}(\rho(u),x_0)\rho(v),x_2).
\end{align*}
Using this we obtain
$$(x_1-x_2)^k \mathfrak D_{Y_V^H}^\rho (Y)(u,x_1)\mathfrak D_{Y_V^H}^\rho (Y)(v,x_2)\in \te{Hom}(V,V((x_1,x_2)))$$
(recalling (\ref{2.35})) and
\begin{align*}
  & \left((x_1-x_2)^k \mathfrak D_{Y_V^H}^\rho (Y)(u,x_1)\mathfrak D_{Y_V^H}^\rho (Y)(v,x_2)\right)|_{x_1=x_2+x_0}\\
  = \ &\left((x_1-x_2)^k Y_V^{\sharp}(\rho(u),x_1)Y_V^{\sharp}(\rho(v),x_2)\right)|_{x_1=x_2+x_0}\\
 =\ &x_0^kY_V^{\sharp}(Y^\sharp(\rho(u),x_0)\rho(v),x_2)\\
  =\  & x_0^k Y_V^{\sharp}(\rho(\mathfrak D_{Y_V^H}^\rho (Y)(u,x_0)v),x_2)\\
  =\  &x_0^k \mathfrak D_{Y_V^H}^\rho (Y)(\mathfrak D_{Y_V^H}^\rho (Y)(u,x_0)v,x_2) .
\end{align*}
 The vacuum  and creation properties follow immediately from
the counit property and the vacuum property (\ref{eq:mod-va-for-vertex-bialg2}).
Therefore, $(V,\mathfrak D_{Y_V^H}^\rho (Y),\vac)$ carries the structure of
a nonlocal vertex algebra  and  by (\ref{rho-homomorphism}) $\rho$
is a nonlocal vertex algebra homomorphism from $\mathfrak D_{Y_V^H}^\rho (V)$
to $V\sharp H$.
\end{proof}

Furthermore, we have:

\begin{prop}\label{right-comodule-algebra}
The nonlocal vertex algebra $\mathfrak D_{Y_V^H}^{\rho}(V)$ obtained in Theorem \ref{prop:deform-va}
with the same map $\rho$ is also a right $H$-comodule nonlocal vertex algebra.
\end{prop}

\begin{proof}  We only need to prove that  $\rho: V\rightarrow V\ot H$
is a homomorphism of nonlocal vertex algebras from
$\mathfrak D_{Y_V^H}^\rho (V)$ to $ \mathfrak D_{Y_V^H}^\rho (V)\ot H$.
For $u,v\in V$, we have
\begin{eqnarray*}
\rho\( \mathfrak D_{Y_V^H}^\rho (Y)(u,x)v\)&=&\rho\( \sum Y(u_{(1)},x)Y_{V}^H(u_{(2)},x)v\)\\
&=&\sum Y^{\ot}(\rho(u_{(1)}),x)\rho\(Y_{V}^H(u_{(2)},x)v\)\\
&=&\sum Y^{\ot}(\rho(u_{(1)}),x)(Y_{V}^H(u_{(2)},x)\ot 1)(v_{(1)}\ot v_{(2)})\\
&=&\sum Y(u_{(1),(1)},x)Y_{V}^H(u_{(2)},x)v_{(1)}\ot Y(u_{(1),(2)},x)v_{(2)}.
\end{eqnarray*}
On the other hand, denoting by $Y_{\ot}'(\cdot,x)$ the vertex operator map of $\mathfrak D_{Y_V^H}^\rho (V)\ot H$,
we have
\begin{eqnarray*}
Y_{\ot}'(\rho(u),x)\rho(v)&=&\sum Y_{\ot}'(u_{(1)}\ot u_{(2)},x) (v_{(1)}\ot v_{(2)})\\
&=&\sum  \mathfrak D_{Y_V^H}^\rho (Y)(u_{(1)},x)v_{(1)}\ot Y(u_{(2)},x)v_{(2)}\\
&=&\sum  Y(u_{(1),(1)},x)Y_{V}^H(u_{(1),(2)},x)v_{(1)}\ot Y(u_{(2)},x)v_{(2)}.
\end{eqnarray*}
Meanwhile, since $H$ is cocommutative,  (\ref{eP23-rho-rho}) yields
$$\sum u_{(1),(1)}\ot u_{(2)}\ot u_{(1),(2)}=\sum u_{(1),(1)}\ot u_{(1),(2)}\ot u_{(2)}.$$
Consequently, we get
\begin{eqnarray}
\rho\( \mathfrak D_{Y_V^H}^\rho (Y)(u,x)v\)=Y_{\ot}'(\rho(u),x)\rho(v).
\end{eqnarray}
This proves that  $\rho$ is also a homomorphism of nonlocal vertex algebras from
$\mathfrak D_{Y_V^H}^\rho (V)$ to $ \mathfrak D_{Y_V^H}^\rho (V)\ot H$, concluding the proof.
\end{proof}

At the end of this section, we present some technical results. First,  by a straightforward argument  we have:

\begin{lem}
Let $H$ be a nonlocal vertex bialgebra
and let $(V,Y_V^H)$ be an $H$-module nonlocal vertex algebra.
Suppose that $(W,Y_W^V)$ is a $V$-module and $(W,Y_W^H)$ is an $H$-module such that
\begin{align}
&Y_W^H(h,x)w\in W\ot \C((x)),\\
  &Y_W^H(h,x_1)Y_W^V(v,x_2)w=\sum Y_W^V(Y_V^H(h_{(1)},x_1-x_2)v,x_2)Y_W^H(h_{(2)},x_1)w
\end{align}
for $h\in S$, $v\in V$, $w\in W$, where $S$ is a generating subset of $H$
as a nonlocal vertex algebra.
Then the two relations above hold for all  $h\in H$, $v\in V$, $w\in W$.
\end{lem}

The following is the second technical result:

\begin{lem}\label{lem:Y-M-rho-compatible}
Let $H$ be a nonlocal vertex bialgebra and let $V$ be a nonlocal vertex algebra.
Suppose that $Y_V^H(\cdot,x):H\rightarrow \te{Hom}(V,V\ot \C((x)))$ is an $H$-module nonlocal
vertex algebra structure
and $\rho:V\rightarrow V\ot H$ is an $H$-comodule nonlocal vertex algebra structure,
satisfying
\begin{align}
&Y_V^H(h,x)v\in T\ot \C((x)),\label{strong-truncation-1}\\
&\rho(Y_V^H(h,x)v)=(Y_V^H(h,x)\ot 1)\rho(v)\label{eq:mod-comod-compatible}
\end{align}
for $h\in S,\   v\in T$, where $S$ and $T$ are generating subspaces of $H$ and $V$
as nonlocal vertex algebras, respectively.
Then $Y_V^H$ and $\rho$ are compatible.
\end{lem}

\begin{proof}  First, we prove that (\ref{strong-truncation-1}) holds for all $h\in H,\  v\in T$.
Let $K$ consist of $h\in H$ such that $Y_V^H(h,x)v\in T\ot \C((x))$ for $v\in T$.
Recall from Lemma \ref{lem:strong-asso} that for any $a,b\in H,\ w\in V$,  we have
$$Y_V^H(Y(a,x_0)b,x_2)w=Y_V^H(a,x_2+x_0)Y_V^H(b,x_2)w.$$
Using this we get that $a_mb\in K$ for $a,b\in K,\ m\in \Z$.
It follows that $K$ is a nonlocal vertex subalgebra of $H$, containing $S$.
Consequently, we have $K=H$, confirming our assertion.

Second, we show that \eqref{eq:mod-comod-compatible} holds for all $h\in H,\  v\in T$.
Let $H'$ consist of $a\in H'$ such that
\eqref{eq:mod-comod-compatible} holds for all $v\in T$.
Using Lemma \ref{lem:strong-asso} and the first assertion,
we can straightforwardly show that $H'$ is a nonlocal vertex subalgebra of $H$.
Consequently, we have  $H'=H$, confirming the second assertion.

Third, we prove that \eqref{eq:mod-comod-compatible} holds for all $h\in H,\ v\in V$.
 Similarly, let $V'$ consist of $v\in V$ such that
\eqref{eq:mod-comod-compatible} holds for all $h\in H$ and we then prove $V'=V$ by showing that
$V'$ is a nonlocal vertex subalgebra containing $T$.
The closure of $V'$ can be established straightforwardly by using (\ref{eq:mod-va-for-vertex-bialg3}) and the fact that
$\rho: V\rightarrow V\ot H$ is a homomorphism of nonlocal vertex algebras.
Now that \eqref{eq:mod-comod-compatible} holds for all $h\in H,\  v\in V$, $Y_V^H$ and $\rho$ are compatible.
\end{proof}

\section{More on nonlocal vertex algebras $\mathfrak D_{Y_{V}^{H}}^\rho (V)$}
In this section, we continue studying the deformed nonlocal vertex algebra
 $\mathfrak D_{Y_{V}^{H}}^\rho (V)$. As the main result, we prove that
$\mathfrak D_{Y_{V}^{H}}^\rho (V)$ is a quantum vertex algebra under the condition that
$H$ is cocommutative, $V$ is a vertex algebra, and $Y_V^H$ is invertible in  a certain sense.

Throughout this section, we assume that {\em $H$ is a cocommutative nonlocal vertex bialgebra.}
By coassociativity we have
$$(\Delta\ot \Delta)\Delta=(\Delta\ot 1\ot 1)(1\ot \Delta)\Delta=(\Delta\ot 1\ot 1)(\Delta\ot 1)\Delta=(1\ot \Delta\ot 1)(\Delta\ot 1)\Delta,$$
As $P_{23}(1\ot \Delta\ot 1)=(1\ot \Delta\ot 1)$ by cocommutativity, consequently we have
\begin{align}\label{P23-three-factor}
P_{23}(\Delta\ot \Delta)\Delta=(\Delta\ot \Delta)\Delta.
\end{align}

\begin{de}\label{de:all-compatible-Y-M}
Let $(V,\rho)$ be a right $H$-comodule nonlocal vertex algebra.
Denote by $\mathfrak L^\rho_H(V)$ the set of all $H$-module nonlocal vertex algebra
structures on $V$ which are compatible with $\rho$.
\end{de}

\begin{ex}
Let $(V,\rho)$ be an $H$-comodule nonlocal vertex algebra. Define
a linear map $Y_{M}^\varepsilon(\cdot,x): H\rightarrow \te{End}_{\C} (V)\subset (\te{End}V)[[x,x^{-1}]] $ by
\begin{align}
Y_{M}^{\varepsilon}(h,x)v=\varepsilon(h)v\    \   \  \mbox{ for }h\in H,\  v\in V.
\end{align}
Then it is straightforward to show that $Y_{M}^\varepsilon \in\mathfrak L_H^\rho(V)$ and
$\mathfrak D_{Y_{M}^\varepsilon}^\rho (V)=V$.
\end{ex}
%

From definition, $\mathfrak L^\rho_H(V)$ is a subset of the space $\te{Hom}\left(H,\te{Hom}(V,V\ot\C((x)))\right)$.
Note that $\te{Hom}(V,V\ot\C((x)))$ is naturally an associative algebra. Furthermore,  with $H$ a coalgebra
$\te{Hom}\left(H,\te{Hom}(V,V\ot\C((x)))\right)$ is an associative algebra with respect to the convolution.
More generally, it is a classical fact that for any associative algebra $A$ with identity,
$\te{Hom}(H,A)$ is an associative algebra  with the operation $*$ defined by
\begin{align}
(f*g)(h)=\sum f(h_{(1)})g(h_{(2)})
\end{align}
for $h\in H,\ f,g\in \te{Hom}(H,A)$, where $\Delta(h)=\sum h_{(1)}\otimes h_{(2)}$, and with $\varepsilon$ as identity.

For $$Y_M,Y_M'\in \te{Hom}(H,\te{Hom}(V,V\ot\C((x)))),$$
we say that $Y_M$ and $Y_M'$ {\em commute} if
\begin{align}
Y_{M}(h,x)Y_{M}'(k,z)=Y_{M}'(k,z)Y_{M}(h,x)\   \   \mbox{ for all }h,k\in H.
\end{align}
Furthermore, a subset $U$ of $\mathfrak L_H^\rho(V)$ is said to be {\em commutative}
if any two elements of $U$ commute.

\begin{prop}\label{lem:L-H-rho-V-compostition}
Let $(V,\rho)$ be a right $H$-comodule nonlocal vertex algebra. For
$$Y_M(\cdot,x),\ Y_M'(\cdot,x)\in \te{Hom}\left(H,\te{Hom}(V,V\ot\C((x)))\right),$$
define a linear map $(Y_M\ast Y_M')(\cdot,x):H\to  \te{Hom}(V,V\ot\C((x)))$ by
\begin{align*}
  &(Y_M\ast Y_M')(h,x)\mapsto \sum Y_M(h_{(1)},x)Y'_M(h_{(2)},x)\quad\te{for }h\in H.
  \end{align*}
Then  $\te{Hom}\left(H,\te{Hom}(V,V\ot\C((x)))\right)$ equipped with the operation $*$ is an associative algebra
 with $Y_{M}^\varepsilon$  as identity.
 Furthermore,  if $Y_M,Y_M'\in \mathfrak L_H^\rho(V)$ and if $Y_M$ and $Y_M'$ commute, then
$Y_M\ast Y_M'\in \mathfrak L_H^\rho(V)$ and $Y_M\ast Y_M'=Y_M'\ast Y_M$.
\end{prop}

\begin{proof} The first assertion immediately follows from the aforementioned general fact,
so it remains to prove the second assertion.
Let $(Y_M,Y_M' )$ be a commuting pair in $\mathfrak L_H^\rho(V)$.
It follows immediately from the cocommutativity of $H$ and the commutativity of $Y_{M}$ and $Y_{M}'$
that $Y_M\ast Y_M'=Y_M'\ast Y_M$.
For $h,k\in H$, using the homomorphism property of $\Delta$, Lemma \ref{lem:strong-asso},
and the commutativity of $Y_{M}$ with $Y_{M}'$, we have
\begin{align*}
  &(Y_M\ast Y_M')(Y(h,z)k,x)\\
  =\  &\sum Y_M\Big( (Y(h,z)k)_{(1)},x\Big)Y_M'\Big( (Y(h,z)k)_{(2)},x\Big)\\
 =\  &\sum Y_M\Big(Y(h_{(1)},z)k_{(1)},x\Big) Y_M'\Big(Y(h_{(2)},z)k_{(2)},x\Big)\\
  =\  &\sum Y_M(h_{(1)},x+z)Y_M(k_{(1)},x)Y_M'(h_{(2)},x+z)Y_M'(k_{(2)},x)\\
 =\  &\sum Y_M(h_{(1)},x+z)Y_M'(h_{(2)},x+z)Y_M(k_{(1)},x)Y_M'(k_{(2)},x)\\
  =\  &(Y_M\ast Y_M')(h,x+z) (Y_M\ast Y_M')(k,x).
\end{align*}
Then it follows that $Y_M\ast Y_M'$ is an $H$-module structure on $V$.
On the other hand,  for $h\in H,\  u,v\in V$,
using the relation
$P_{23}(\Delta\ot \Delta)\Delta(h)=(\Delta\ot \Delta)\Delta(h)$ from (\ref{P23-three-factor}),
we get
\begin{align*}
&(Y_M\ast Y_M')(h,x)Y(u,z)v\\
=\  &\sum Y_M(h_{(1)},x)Y_M'(h_{(2)},x)Y(u,z)v\\
=\  &\sum Y_M(h_{(1)},x)Y(Y_M'(h_{(2),(1)},x-z)u,z)Y_M'(h_{(2),(2)},x)v\\
=\  &\sum Y\Big(Y_M(h_{(1),(1)},x-z)Y_M'(h_{(2),(1)},x-z)u,z\Big)
   Y_M(h_{(1),(2)},x)Y_M'(h_{(2),(2)},x)v\\
=\ &\sum Y\Big(Y_M(h_{(1),(1)},x-z)Y_M'(h_{(1),(2)},x-z)u,z\Big)
    Y_M(h_{(2),(1)},x)Y_M'(h_{(2),(2)},x)v\\
=\  &\sum Y\Big((Y_M\ast Y_M')(h_{(1)},x-z)u,z\Big)
    (Y_M\ast Y_M')(h_{(2)},x)v.
\end{align*}
This proves that $(V,Y_M\ast Y_M')$ is an $H$-module nonlocal vertex algebra.
It is straightforward to show that $Y_M\ast Y_M'$  is compatible with $\rho$.
Therefore, $Y_M\ast Y_M' \in \mathfrak L_H^\rho(V)$.
\end{proof}

Furthermore, we have:

\begin{prop}\label{prop:D-Y-M-rho-composition}
Let $(V,\rho)$ be a right $H$-comodule nonlocal vertex algebra and
let $(Y_M,Y_M' )$ be a commuting pair in $\mathfrak L_H^\rho(V)$.
Then $Y_M\in\mathfrak L_H^\rho\(\mathfrak D_{Y_M'}^\rho (V)\)$ and
\begin{align}
  \mathfrak D_{Y_M}^\rho\(\mathfrak D_{Y_M'}^\rho (V)\)
  =\mathfrak D_{Y_M\ast Y_M'}^\rho (V).
\end{align}
\end{prop}

\begin{proof}
Let $h\in H,\  u\in V$. Recall that $\rho: V\rightarrow V\ot H$ is an $H$-module homomorphism, where
$\rho(v)=\sum v_{(1)}\otimes v_{(2)}\in V\otimes H$ for $v\in V$.
Then we have
\begin{align*}
  &Y_M(h,x)\mathfrak D_{Y_M'}^\rho (Y)(u,z)
  =\sum Y_M(h,x)Y(u_{(1)},z)Y_M'(u_{(2)},z)\\
=& \sum Y(Y_M(h_{(1)},x-z)u_{(1)},z)Y_M(h_{(2)},x)Y_M'(u_{(2)},z)\\
=& \sum Y(Y_M(h_{(1)},x-z)u_{(1)},z)Y_M'(u_{(2)},z)Y_M(h_{(2)},x)\\
=& \sum Y\(\(Y_M(h_{(1)},x-z)u\)_{(1)},z\)
Y_M'\(\(Y_M(h_{(1)},x-z)u\)_{(2)},z\)Y_M(h_{(2)},x)\\
=& \sum \mathfrak D_{Y_M'}^\rho (Y)\(Y_M(h_{(1)},x-z)u,z\)Y_M(h_{(2)},x).
\end{align*}
Thus, $(\mathfrak D_{Y_M'}^\rho (V),Y_M)$ is an $H$-module nonlocal vertex algebra.

From Proposition \ref{right-comodule-algebra}, $\mathfrak D_{Y_M'}^\rho (V)$ with map $\rho$ is a right $H$-comodule
nonlocal vertex algebra. As $Y_M$ is compatible with $\rho$,
we have $Y_M\in\mathfrak L_H^\rho\(\mathfrak D_{Y_M'}^\rho (V)\)$.

Let $u\in V$.  Note that from the cocommutativity we have
\begin{align*}
\sum u_{(1),(1)}\ot u_{(1),(2)}\ot u_{(2)}=
\sum u_{(1)}\ot u_{(2),(2)}\ot u_{(2),(1)}.
\end{align*}
Then  using the commutativity of $Y_{M}$ and $Y_{M}'$, we get
\begin{align*}
&\mathfrak D^\rho_{Y_M}\left(\mathfrak D^\rho_{Y_M'}(Y)\right)(u,x)\\
  =\ &\sum \mathfrak D^\rho_{Y_M'}(Y)(u_{(1)},x)Y_M(u_{(2)},x)\\
  =\ &\sum Y(u_{(1),(1)},x)Y_M'(u_{(1),(2)},x)Y_M(u_{(2)},x)\\
  =\ &\sum Y(u_{(1)},x)Y_M'(u_{(2),(2)},x)Y_M(u_{(2),(1)},x)\\
  =\ &\sum Y(u_{(1)},x)Y_M(u_{(2),(1)},x)Y_M'(u_{(2),(2)},x)\\
  =\ &\mathfrak D_{Y_M\ast Y_M'}^\rho (Y)(u,x).
\end{align*}
Therefore, we have
$\mathfrak D_{Y_M}^\rho\(\mathfrak D_{Y_M'}^\rho (V)\)=\mathfrak D_{Y_M\ast Y_M'}^\rho (V)$.
\end{proof}

We also have the following result which is straightforward to prove:

\begin{lem}\label{lem:Y-deformed-Y-M-invertible}
Let $Y_M$ be an invertible element of $\mathfrak L_H^\rho(V)$ with inverse $Y_M\inverse$
with respect to the operation $*$.  Then
\begin{align}
  Y(u,x)v=\sum \mathfrak D_{Y_M}^\rho (Y)\(u_{(1)},x\)Y_M\inverse\(u_{(2)},x\)v
\   \   \quad\te{for }u,v\in V,
\end{align}
where $\rho(u)=\sum u_{(1)}\ot u_{(2)}\in V\ot H$.
\end{lem}
%


As the main result of this section, we have:

\begin{thm}\label{thm:S-op}
Assume that $H$ is a nonlocal vertex cocommutative bialgebra.
Let $V$ be a vertex algebra with a right $H$-comodule vertex algebra structure $\rho$.
Suppose that
$\{Y_M^{\pm 1}\}$ is a commutative subset of $\mathfrak L_H^\rho(V)$ with $Y_M$ and $Y_M^{-}$ inverses each other
   with respect to the operation $*$.
Define a linear map $\cS(x): V\otimes V\rightarrow V\otimes V\ot \C((x))$ by
\begin{align}
  \cS(x)(v\ot u)=\sum Y_M(u_{(2)},-x)v_{(1)}\ot Y_M\inverse (v_{(2)},x)u_{(1)}
\end{align}
for $u,v\in V$.
Then $\cS(x)$ is a unitary rational quantum Yang-Baxter operator and
the nonlocal vertex algebra $\mathfrak D_{Y_M}^\rho (V)$
which was obtained in Theorem \ref{prop:deform-va} with $\cS(x)$ is a quantum vertex algebra.
\end{thm}

\begin{proof}  Let $u,v\in V$. Note that as $V$ is a vertex algebra, the usual skew symmetry holds.
As $\rho: V\rightarrow V\ot H$ is an $H$-module homomorphism, we have
$$\sum  \(Y_M(u_{(2)},x)v\)_{(1)}\ot \(Y_M(u_{(2)},x)v\)_{(2)}=\sum Y_M(u_{(2)},x)v_{(1)} \ot  v_{(2)}.   $$
By using these properties and Lemma \ref{lem:Y-deformed-Y-M-invertible},
it is straightforward to show that
\begin{align*}
  &\mathfrak D_{Y_M}^\rho (Y)(u,x)v
 =\  e^{x\D}\mathfrak D_{Y_M}^\rho (Y)(-x)\cS(-x)(v\ot u).
\end{align*}
It then follows from \cite{li-nonlocal} that  $\cS$-locality holds.

Recall that $\cS^{21}(x)=R_{12} \cS(x)R_{12}$.
For any $a,b\in V$, we have
\begin{align*}
&\cS^{21}(-x)(b\ot a)
=\sum Y_M\inverse (a_{(2)},-x)b_{(1)} \ot Y_M(b_{(2)},x)a_{(1)}.
\end{align*}
On the other hand,  with $(\rho \ot 1)\rho=(1\ot \Delta)\rho$, it follows that
\begin{eqnarray*}
&&\sum u_{(1),(2)}\ot u_{(2)}\ot u_{(1),(1)}=\sum u_{(2)(1)}\ot u_{(2),(2)}\ot u_{(1)},\\
&&\sum v_{(1),(1)}\ot v_{(1),(2)}\ot v_{(2)}=\sum v_{(1)}\ot v_{(2),(1)}\ot v_{(2),(2)}.
\end{eqnarray*}
The compatibilities of $\rho$ and $Y_M^{\pm 1}$ yield
$$ \sum \(Y_{M}(u_{(2)},-x)v_{(1)}\)_{(1)}\ot \(Y_{M}(u_{(2)},-x)v_{(1)}\)_{(2)}
= \sum Y_{M}(u_{(2)},-x)v_{(1)(1)}\ot v_{(1)(2)},  $$
$$\sum \(Y_M\inverse (v_{(2)},x)u_{(1)}\)_{(2)} \ot \(Y_M\inverse (v_{(2)},x)u_{(1)}\)_{(1)}
= \sum u_{(1)(2)}\ot Y_M\inverse (v_{(2)},x)u_{(1)(1)}.  $$
Then using all of these relations we get
\begin{align*}
  &\cS^{21}(-x)\cS(x)(v\ot u)\\
  =&\sum \cS^{21}(-x)\(Y_M(u_{(2)},-x)v_{(1)}\ot Y_M\inverse (v_{(2)},x)u_{(1)}\)\\
  =&\sum Y_M\inverse\(\(Y_M\inverse (v_{(2)},x)u_{(1)}\)_{(2)},-x\)
  \(Y_M(u_{(2)},-x)v_{(1)}\)_{(1)}\\
  &\quad\ot
  Y_M\(
    \(Y_M(u_{(2)},-x)v_{(1)}\)_{(2)},x
  \)\(Y_M\inverse (v_{(2)},x)u_{(1)}\)_{(1)}\\
  =&\sum Y_M\inverse\( u_{(1),(2)},-x\)Y_M(u_{(2)},-x)v_{(1),(1)}\ot
  Y_M\(v_{(1),(2)},x\)Y_M\inverse (v_{(2)},x)u_{(1),(1)}\\
  =&\sum Y_M\inverse\( u_{(2),(1)},-x\)Y_M(u_{(2),(2)},-x)v_{(1)}\ot
  Y_M\(v_{(2),(1)},x\)Y_M\inverse (v_{(2),(2)},x)u_{(1)}\\
  =&\sum \varepsilon(u_{(2)})v_{(1)}\ot \varepsilon(v_{(2)})u_{(1)}
  =\sum \varepsilon(v_{(2)})v_{(1)}\ot \varepsilon(u_{(2)})u_{(1)}
  =v\ot u.
\end{align*}
This proves that $\cS(x)$ is unitary.

Furthermore, for $u,v,w\in V$, as $\rho$ is a vertex algebra homomorphism we have
\begin{align*}
&\cS(x)\(\mathfrak D_{Y_M}^\rho (Y)(u,z)v\ot w\)\\
=&\sum \cS(x)\(Y(u_{(1)},z)Y_M(u_{(2)},z)v\ot w\)\\
=&\sum Y_M \(w_{(2)},-x\)\(Y(u_{(1)},z)Y_M(u_{(2)},z)v\)_{(1)}
    \ot Y_M\inverse\(\(Y(u_{(1)},z)Y_M(u_{(2)},z)v\)_{(2)},x\)w_{(1)}\\
=&\sum Y_M \(w_{(2)},-x\)Y(u_{(1),(1)},z)(Y_M(u_{(2)},z)v)_{(1)}
    \ot Y_M\inverse \(Y(u_{(1),(2)},z)(Y_M(u_{(2)},z)v)_{(2)},x\)w_{(1)}\\
=&\sum Y_M \(w_{(2)},-x\)Y(u_{(1),(1)},z)Y_M(u_{(2)},z)v_{(1)}
    \ot Y_M\inverse\(Y(u_{(1),(2)},z)v_{(2)},x\)w_{(1)}\\
=&\sum Y( Y_M\( w_{(2),(1)},-x-z \) u_{(1),(1)},z)
    Y_M\(w_{(2),(2)},-x\)Y_M(u_{(2)},z)v_{(1)} \\
&\quad\ot
    Y_M\inverse\(u_{(1),(2)}, x+z\)Y_M\inverse\(v_{(2)},x\)w_{(1)}.
\end{align*}
On the other hand,  we have
\begin{align*}
&\(\mathfrak D_{Y_M}^\rho (Y)(z)\ot 1\)\cS^{23}(x)\cS^{13}(x+z)(u\ot v\ot w)\\
=&\(\mathfrak D_{Y_M}^\rho (Y)(z)\ot 1\)\cS^{23}(x)
\sum Y_M(w_{(2)},-x-z)u_{(1)}\ot v\ot Y_M\inverse (u_{(2)},x+z)w_{(1)}\\
=&\(\mathfrak D_{Y_M}^\rho (Y)(z)\ot 1\)
\sum Y_M(w_{(2)},-x-z)u_{(1)}\ot
Y_M \(\(Y_M\inverse (u_{(2)},x+z)w_{(1)}\)_{(2)},-x\)v_{(1)}\\
&\quad
    \ot Y_M\inverse\(v_{(2)},x\)\(Y_M\inverse (u_{(2)},x+z)w_{(1)}\)_{(1)}\\
=&\(\mathfrak D_{Y_M}^\rho (Y)(z)\ot 1\)
\sum Y_M(w_{(2)},-x-z)u_{(1)}\ot
Y_M \(w_{(1),(2)},-x\)v_{(1)}\\
&\quad
    \ot Y_M\inverse\(v_{(2)},x\)Y_M\inverse (u_{(2)},x+z)w_{(1),(1)}\\
=&\sum Y\(\(Y_M(w_{(2)},-x-z)u_{(1)}\)_{(1)},z\)
    Y_M\(\(Y_M(w_{(2)},-x-z)u_{(1)}\)_{(2)},z  \)
    Y_M \(w_{(1),(2)},-x\)v_{(1)}
\\
&\quad
    \ot Y_M\inverse\(v_{(2)},x\)Y_M\inverse (u_{(2)},x+z)w_{(1),(1)}\\
=&\sum Y\(Y_M(w_{(2)},-x-z)u_{(1),(1)},z\)
    Y_M\(u_{(1),(2)},z  \)
    Y_M \(w_{(1),(2)},-x\)v_{(1)}\\
&\quad
    \ot Y_M\inverse\(v_{(2)},x\)Y_M\inverse (u_{(2)},x+z)w_{(1),(1)}\\
=&\sum Y\(Y_M(w_{(2)},-x-z)u_{(1),(1)},z\)
    Y_M\(u_{(1),(2)},z  \)
    Y_M \(w_{(1),(2)},-x\)v_{(1)}\\
&\quad
    \ot Y_M\inverse\(v_{(2)},x\)Y_M\inverse (u_{(2)},x+z)w_{(1),(1)}\\
=&\sum Y\(Y_M(w_{(2)},-x-z)u_{(1),(1)},z\)
    Y_M \(w_{(1),(2)},-x\)Y_M\(u_{(1),(2)},z  \)v_{(1)}\\
&\quad
    \ot Y_M\inverse (u_{(2)},x+z)Y_M\inverse\(v_{(2)},x\)w_{(1),(1)}\\
=&\sum Y\(Y_M(w_{(2),(1)},-x-z)u_{(1),(1)},z\)
    Y_M \(w_{(2),(2)},-x\)Y_M\(u_{(2)},z  \)v_{(1)}\\
&\quad
    \ot Y_M\inverse (u_{(1),(2)},x+z)Y_M\inverse\(v_{(2)},x\)w_{(1)}.
\end{align*}
For the last equality we are also using the propertites
\begin{align*}
&\sum  w_{(2)}\ot w_{(1),(2)}\ot w_{(1),(1)}=\sum w_{(2),(1)}\ot w_{(2),(2)}\ot w_{(1)},\\
&\sum u_{(1),(1)} \ot u_{(1),(2)}\ot  u_{(2)}=\sum u_{(1),(1)}\ot u_{(2)}\ot u_{(1),(2)},
\end{align*}
i.e., $P_{13}(\rho\ot 1)\rho(w)=P_{13}P_{23}(1\ot \Delta)\rho(w)$, $\  P_{23}(\rho\ot 1)\rho(u)=(\rho\ot 1)\rho(u)$,
which follow from the cocommutativity as before.
Thus we have
$$\cS(x)\(\mathfrak D_{Y_M}^\rho (Y)(u,z)v\ot w\)=\(\mathfrak D_{Y_M}^\rho (Y)(z)\ot 1\)\cS^{23}(x)\cS^{13}(x+z)(u\ot v\ot w).$$
This proves that the hexagon identity holds.
%

Next, we show that the quantum Yang-Baxter equation holds. Let $u,v,w\in V$.
Recall that the compatibilities of $\rho$ with $Y_M$ and $Y_{M}^{-1}$ state that
$$\sum (Y_M^{\pm 1}(h,z)a)_{(1)}\ot (Y_M^{\pm 1}(h,z)a)_{(2)}=\sum Y_M^{\pm 1}(h,z)a_{(1)}\ot a_{(2)}$$
for $h\in H,\ a\in V$. Using these relations we get
\begin{align*}
&\cS^{12}(x)\cS^{13}(x+z)\cS^{23}(z)(u\ot v\ot w)\\
=&\sum \cS^{12}(x)\cS^{13}(x+z)\(u\ot
Y_M(w_{(2)},-z)v_{(1)}\ot Y_M\inverse (v_{(2)},z)w_{(1)}\)\\
=&\sum \cS^{12}(x)
Y_M\(\(Y_M\inverse (v_{(2)},z)w_{(1)}\)_{(2)},-x-z\)u_{(1)}\ot
Y_M(w_{(2)},-z)v_{(1)}\\
&\quad\ot
Y_M\inverse (u_{(2)},x+z)\(Y_M\inverse (v_{(2)},z)w_{(1)}\)_{(1)}\\
=&\sum \cS^{12}(x)
Y_M\(w_{(1),(2)},-x-z\)u_{(1)}\ot
Y_M(w_{(2)},-z)v_{(1)}\\
&\quad\ot
Y_M\inverse (u_{(2)},x+z)Y_M\inverse (v_{(2)},z)w_{(1),(1)}\\
=&\sum Y_M\(\(Y_M(w_{(2)},-z)v_{(1)}\)_{(2)},-x\)
\(Y_M\(w_{(1),(2)},-x-z\)u_{(1)}\)_{(1)}\\
&\quad\ot Y_M\inverse \(\(Y_M\(w_{(1),(2)},-x-z\)u_{(1)}\)_{(2)},x\)\(Y_M(w_{(2)},-z)v_{(1)}\)_{(1)}\\
&\quad
\ot   Y_M\inverse (u_{(2)},x+z)Y_M\inverse (v_{(2)},z)w_{(1),(1)}\\
=&\sum Y_M\(v_{(1),(2)},-x\)
Y_M\(w_{(1),(2)},-x-z\)u_{(1),(1)}\\
&\quad\ot Y_M\inverse \(u_{(1),(2)},x\)Y_M(w_{(2)},-z)v_{(1),(1)}\\
&\quad
\ot   Y_M\inverse (u_{(2)},x+z)Y_M\inverse (v_{(2)},z)w_{(1),(1)}.
\end{align*}
Similarly, we have
\begin{align*}
&\cS^{23}(z)\cS^{13}(x+z)\cS^{12}(x)(u\ot v\ot w)\\
=&\sum Y_M\(w_{(2)},-x-z\)Y_M(v_{(2)},-x)u_{(1),(1)}
\ot
Y_M\(w_{(1),(2)},-z\)Y_M\inverse (u_{(2)},x)v_{(1),(1)}\\
&\quad
\ot Y_M\inverse \(v_{(1),(2)},z\)Y_M\inverse \(u_{(1),(2)},x+z\)w_{(1),(1)}\\
=&\sum Y_M(v_{(2)},-x)Y_M\(w_{(2)},-x-z\)u_{(1),(1)}
\ot
Y_M\inverse (u_{(2)},x)Y_M\(w_{(1),(2)},-z\)v_{(1),(1)}\\
&\quad
\ot Y_M\inverse \(u_{(1),(2)},x+z\)Y_M\inverse \(v_{(1),(2)},z\)w_{(1),(1)},
\end{align*}
where for the last equality we are also using the commutativity for $(Y_M^{\pm 1},Y_{M}^{\pm 1})$ and $(Y_M^{\pm 1}, Y_M^{\mp 1})$.
Note that with $H$ cocommutative we have
$P_{12}(\rho\ot 1)\rho(v)=P_{12}P_{23}(\rho\ot 1)\rho(v),$
which states
$$\sum v_{(1),(2)}\ot v_{(1),(1)}\ot v_{(2)}=\sum v_{(2)}\ot v_{(1),(1)}\ot  v_{(1),(2)}.$$
Similarly, we have 
\begin{align*}
&\sum u_{(1),(1)}\ot u_{(1),(2)} \ot u_{(2)}  =\sum u_{(1),(1)}\ot u_{(2)} \ot u_{(1),(2)},\nonumber\\
&\sum w_{(1),(2)}\ot w_{(2)} \ot w_{(1),(1)}  =\sum w_{(2)} \ot  w_{(1),(2)} \ot w_{(1),(1)}.
\end{align*}
Then we conclude
$$\cS^{12}(x)\cS^{13}(x+z)\cS^{23}(z)(u\ot v\ot w)=\cS^{23}(z)\cS^{13}(x+z)\cS^{12}(x)(u\ot v\ot w).$$

Last, the shift condition can be proved straightforwardly.
Therefore, $\mathfrak D_{Y_M}^\rho (V)$ with $\cS(x)$ is a quantum vertex algebra.
\end{proof}

The following is a result about the generating subsets of $\mathfrak D_{Y_M}^\rho (V)$:

\begin{lem}\label{gen-set}
Let $Y_M$ be an invertible element  of $\mathfrak L_H^\rho(V)$.
Suppose that $S$ and $T$ are generating subspaces of $V$ and $H$ as nonlocal vertex algebras, respectively, such that
\begin{align}
\rho(S)\subset S\ot T,\quad\Delta(T)\subset T\ot T,\quad
  Y_M\inverse(T,x)S\subset S\ot \C((x)).
\end{align}
Then $S$ is also a generating subset of $\mathfrak D_{Y_M}^\rho (V)$.
\end{lem}

\begin{proof}  Let $X$ consist of all subspaces $U$  of $V$ such that
\begin{align*}
  Y_M\inverse (T,x)U\subset U \ot \C((x)).
\end{align*}
Denote by $K$ the sum of all such subspaces. It is clear that  $K$ is the (unique) largest in $X$.
 We see that $S$ and $\C {\bf 1}$ are in $X$, so that  $\C {\bf 1}+S\subset K$.
Let $K^2$ denote the linear span of $u_nv$  in $\mathfrak D_{Y_M}^\rho (V)$ for $u,v\in K,\ n\in \Z$.
For $u,v\in K$ and $t\in T$, from the first part of the proof of Proposition \ref{prop:D-Y-M-rho-composition}, we see that
\begin{align*}
  Y_M\inverse(t,x)\mathfrak D_{Y_M}^\rho (Y)(u,z)v
  =\sum \mathfrak D_{Y_M}^\rho (Y)(Y_M\inverse (t_{(1)},x-z)u,z)Y_M\inverse(t_{(2)},x)v.
\end{align*}
It follows that $K^2$ is in $X$. As $K$ is the largest, we have $K^2\subset K$.
Thus $K$ is a nonlocal vertex subalgebra of $\mathfrak D_{Y_M}^\rho (V)$ with $S\subset K$.

Now, let $H'$ be the maximal subspace of $H$ such that
\begin{align*}
  Y_M\inverse (H',x)K\subset K\ot\C((x)).
\end{align*}
We have  $T+  \C{\bf 1}\subset H'$.  Note that for $h,k\in H'$, we have
\begin{align*}
&Y_M\inverse(Y(h,z)k,x)V'=Y_M\inverse(h,x+z)Y_M\inverse(k,x)K\subset K\ot \C((x))[[z]].
\end{align*}
It follows that  $H'$ is a nonlocal vertex subalgebra of $H$ and hence $H'=H$. Thus $Y_M\inverse (H,x)K\subset K\ot\C((x))$.

Let $V''$ be any subspace of $V$ such that
\begin{align}\label{eq:gen-set-001}
  Y_M(H,x)V''\subset V''\ot \C((x))\quad\te{and}\quad
  \rho(V'')\subset V''\ot H.
\end{align}
That is, $V''$ is an $H$-submodule and a sub-comodule.
For $u,v\in V''$, we have
\begin{align*}
&\rho\(\mathfrak D_{Y_M}^\rho (V)(u,z)v\)=Y^\sharp(\rho(u),z)\rho(v)\\
=&\sum Y(u_{(1)},z)Y_M(u_{(2),(1)},z)v_{(1)}\ot Y(u_{(2),(2)},z)v_{(2)}\\
=&\sum Y(u_{(1),(1)},z)Y_M(u_{(1),(2)},z)v_{(1)}\ot Y(u_{(2)},z)v_{(2)}\\
=& \sum \mathfrak D_{Y_M}^\rho (Y)(u_{(1)},z)v_{(1)}\ot Y(u_{(2)},z)v_{(2)}.
\end{align*}
It follows that \eqref{eq:gen-set-001} holds for the nonlocal vertex subalgebra of
$\mathfrak D_{Y_M}^\rho (V)$ generated by $V''$.
Since \eqref{eq:gen-set-001} holds for $V''=S$, it holds for $V''=K$.

From Lemma \ref{lem:Y-deformed-Y-M-invertible}, for $u\in S,\  v\in K$ we have
\begin{align*}
&Y(u,x)v=\sum \mathfrak D_{Y_M}^\rho (Y)(u_{(1)},x)Y_M\inverse(u_{(2)},x)v\in K[[x,x\inverse]].
\end{align*}
As $\C{\bf 1}+S\subset K$, it follows that $K$ contains the nonlocal vertex subalgebra of $V$ generated by $S$.
Thus $K=V$.
Therefore, $S$ is also a generating subspace of the quantum vertex algebra $\mathfrak D_{Y_M}^\rho (V)$.
\end{proof}

\section{$\phi$-coordinated quasi modules for smash product nonlocal vertex algebras}

In this section, we first recall from \cite{li-cmp} the basic notions and results on
$\phi$-coordinated quasi modules for nonlocal vertex algebras and
then study (equivariant) $\phi$-coordinated modules for smash product nonlocal vertex algebras.

We begin with formal group.
A {\em one-dimensional formal group (law)} over $\C$ (see \cite{Ha})
is a formal power series $F(x,y)\in \C[[x,y]]$ such that
$$F(0,y)=y,\  \ F(x,0)=x,\   \    F(F(x,y),z)=F(x,F(y,z)).$$
A particular example  is the {\em additive formal group} $F_a(x,y):=x+y$.

Let $F(x,y)$ be a one-dimensional formal group over $\C$.
An {\em associate}  of $F(x,y)$ (see \cite{li-cmp}) is a formal series
$\phi(x,z)\in\C((x))[[z]]$, satisfying the condition that
\begin{align*}
  \phi(x,0)=x,\ \  \  \
  \phi(\phi(x,y),z)=\phi(x,F(y,z)).
\end{align*}

The following result  was obtained therein:

\begin{prop}\label{prop:classification-assos-1}
Let $p(x)\in \C((x))$.
Set
\begin{align*}
  \phi(x,z)=e^{zp(x)\frac{d}{dx}}x
  =\sum_{n\ge0}\frac{z^n}{n!}\(p(x)\frac{d}{dx}\)^nx\in\C((x))[[z]].
\end{align*}
Then $\phi(x,z)$ is an associate of $F_a(x,y)$.
On the other hand, every associate of $F_a(x,y)$ is of this form with $p(x)$ uniquely determined.
\end{prop}

\begin{ex}
Taking $p(x)=1$ in Proposition \ref{prop:classification-assos-1}, we get
$\phi(x,z)=x+z=F_{a}(x,z)$ (the formal group itself), and taking $p(x)=x$ we get
$\phi(x,z)=xe^{z}$. More generally, for $r\in \Z$,  from \cite{fhl} we have
\begin{eqnarray}\label{expression-phi-r}
\phi_r(x,z):=e^{zx^{r+1}\frac{d}{dx}}x=\begin{cases}x\(1-rzx^{r}\)^{\frac{-1}{r}}&\  \  \mbox{ if }r\ne 0\\
xe^z& \  \  \mbox{ if }r= 0.
\end{cases}
\end{eqnarray}
\end{ex}

From now on, we shall {\em always} assume that $\phi(x,z)$ is an associate of $F_{a}(x,y)$
with $\phi(x,z)\ne x$, or equivalently,
$\phi(x,z)=e^{zp(x)\frac{d}{dx}}x$ with $p(x)\ne 0$.

From \cite{li-cmp}, for $f(x_1,x_2)\in \C((x_1,x_2))$,  $f(\phi(x,z),x)$ exists in $\C((x))[[z]]$ and
 the correspondence
$f(x_1,x)\mapsto f(\phi(x,z),x)$ gives a  ring embedding
$$\iota_{x_1=\phi(x,z)}:\  \C((x_1,x))\rightarrow \C((x))[[z]]\subset \C((x))((z)).$$
Denote by $\C_\ast((x_1,x_2))$ the fraction field of $\C((x_1,x_2))$.
Then $\iota_{x_1=\phi(x,z)}$ naturally extends to a field embedding
\begin{align}
\iota_{x_1=\phi(x,z)}^{*}: \  \C_\ast((x_1,x))\rightarrow \C((x))((z)).
\end{align}
 As a convention,  for $F(x_1,x_2)\in \C_\ast((x_1,x_2))$, we write
\begin{align}
F(\phi(x,z),x)=\iota_{x_1=\phi(x,z)}^{*}(F(x_1,x))\in \C((x))((z)).
\end{align}

For $f(x_1,x_2)\in\C((x_1,x_2))$,  by definition
$$f(\phi(x,z_1),\phi(x,z_2))=\(e^{z_1p(x_1)\partial_{x_1}}
e^{z_2p(x_2)\partial_{x_2}}f(x_1,x_2)\)|_{x_1=x_2=x},$$
which exists in $\C((x))[[z_1,z_2]]$, and
  the correspondence $$f(x_1,x_2)\mapsto f(\phi(x,z_1),\phi(x,z_2))$$
is a ring embedding of $\C((x_1,x_2))$ into $\C((x))[[z_1,z_2]]$ (see \cite{JKLT1}).
Then  this ring embedding gives rise to a field embedding
\begin{align}
 \C_{*}((x_1,x_2))\rightarrow \C((x))_{*}((z_1,z_2)),
\end{align}
where $ \C((x))_{*}((z_1,z_2))$ denotes the fraction field of $\C((x))((z_1,z_2))$. As a convention,
for $F(x_1,x_2)\in \C_{*}((x_1,x_2))$, we
view $F(\phi(x,z_1),\phi(x,z_2))$ as an element of $\C((x))_{*}((z_1,z_2))$.

\begin{de}
Let $V$ be a nonlocal vertex algebra.
A {\em $\phi$-coordinated quasi $V$-module} is a vector space $W$ equipped with a linear map
\begin{eqnarray*}
Y_{W}(\cdot,x): &&  V\rightarrow (\te{End} W)[[x,x^{-1}]]\\
&&v\mapsto Y_{W}(v,x),
\end{eqnarray*}
satisfying the conditions that
$$Y_{W}(u,x)w\in W((x))\   \   \    \te{for }u\in V,\ w\in W,$$
$$Y_{W}({\bf 1},x)=1_{W} \   \  (\te{the identity operator on }W),$$
and that for any $u,v\in V$, there exists $f(x_1,x_2)\in \C((x_1,x_2))^{\times}$ such that
\begin{align}
&f(x_1,x_2)Y_{W}(u,x_1)Y_{W}(v,x_2)\in \te{Hom}(W,W((x_1,x_2))),\\
&\(f(x_1,x_2)Y_{W}(u,x_1)Y_{W}(v,x_2)\)|_{x_1=\phi(x_2,z)}=f(\phi(x_2,z),x_2)Y_{W}(Y(u,z)v,x_2).
\end{align}
\end{de}

The following notion was introduced in \cite{JKLT1} (cf. \cite{li-jmp}):

\begin{de}\label{de:G-equiv-phi-mod-1}
Let $V$ be a $(G,\chi)$-module nonlocal vertex algebra and let $\chi_{\phi}$ be a linear character of $G$
such that
\begin{eqnarray}\label{p(x)-compatibility}
\phi(x,\chi(g)z)=\chi_{\phi}(g)\phi(\chi_{\phi}(g)^{-1}x,z)\   \   \mbox{ for }g\in G.
\end{eqnarray}
 A {\em $(G,\chi_{\phi})$-equivariant $\phi$-coordinated quasi $V$-module} is a $\phi$-coordinated quasi $V$-module
$(W,Y_W)$ satisfying the conditions that
\begin{eqnarray}\label{phi-module-equiv-1}
Y_{W}(R(g)v,x)=Y_{W}(v,\chi_{\phi}(g)^{-1}x)\   \   \   \   \mbox{ for }g\in G,\  v\in V
\end{eqnarray}
and that for $u,v\in V$, there exists $q(x)\in \C_{\chi_{\phi}(G)}[x]$ such that
\begin{eqnarray}
q(x_1/x_2)Y_W(u,x_1)Y_W(v,x_2)\in \te{Hom} (W,W((x_1,x_2))),
\end{eqnarray}
where $\C_{\chi_{\phi}(G)}[x]$ denotes the multiplicative monoid generated in $\C[x]$ by $x-\chi_{\phi}(g)$ for $g\in G$.
\end{de}

\begin{rem}
{\em It was proved (see loc. cit.) that the  condition (\ref{p(x)-compatibility}) is equivalent to
\begin{eqnarray}\label{chi-phi-p(x)}
p(\chi_{\phi}(g)x)=\chi(g)^{-1}\chi_{\phi}(g)p(x)\   \   \mbox{ for }g\in G,
\end{eqnarray}
where $\phi(x,z)=e^{zp(x)\frac{d}{dx}}x$.
In case $p(x)=x^{r+1}$ with $r\in \Z$, the compatibility condition (\ref{p(x)-compatibility})
  is equivalent to $\chi=\chi_{\phi}^{-r}$.}
\end{rem}

The following is a straightforward analogue of Lemma \ref{lem:G-va-generating}
for $(G,\chi_{\phi})$-equivariant $\phi$-coordinated quasi $V$-modules:

\begin{lem}\label{lem:G-equiv-phi-mod-generating}
Under the setting of Definition \ref{de:G-equiv-phi-mod-1}, assume that all the conditions hold except
(\ref{phi-module-equiv-1}), and instead assume
\begin{align}
  Y_W^\phi(R(g)v,x)=Y_W^\phi(v,\chi_{\phi}(g)^{-1}x)\quad {for }\  g\in G,\  v\in S,
\end{align}
where $S$ is a generating subset of $V$.
Then $W$ is a $(G,\chi_{\phi})$-equivariant $\phi$-coordinated quasi $V$-module.
\end{lem}

Next, we study $(G,\chi)$-module nonlocal vertex algebra structures
on smash product nonlocal vertex algebra $V\sharp H$ and
equivariant $\phi$-coordinated quasi modules for $V\sharp H$.
For the rest of this section, we fix a group $G$ with a linear character $\chi$.
We first formulate the following notion:

\begin{de}\label{de:vbialg-G-algebra}
A \emph{$(G,\chi)$-module nonlocal vertex bialgebra} $H$ is a  nonlocal vertex bialgebra
and a $(G,\chi)$-module nonlocal vertex algebra on which $G$ acts as an automorphism group with $H$ viewed as a coalgebra.
\end{de}

The following two propositions can be proved straightforwardly:

\begin{prop}\label{smash-G-algebra}
Let $H$ be a $(G,\chi)$-module nonlocal vertex bialgebra.
Suppose that $V$ is a $(G,\chi)$-module nonlocal vertex algebra and an $H$-module nonlocal vertex algebra
with the $H$-module structure denoted by $Y_V^H(\cdot,x)$ such that
\begin{eqnarray}\label{RV-RH-compatibility}
R_{V}(g)Y_V^H(h,x)v=Y_V^H(R_{H}(g) h,\chi(g)x)R_{V}(g) v\quad { for }\  g\in G,\ h\in H,\ v\in V,
\end{eqnarray}
where $R_{H}$ and $R_{V}$ denote the representations of $G$ on $H$ and $V$, respectively.
Then
$V\sharp H$ is a $(G,\chi)$-module nonlocal vertex algebra with $R=R_V\ot R_H$.
\end{prop}
%


\begin{prop}\label{dual-vertex-G-algebra}
Let $H,V$ be given as in Proposition \ref{smash-G-algebra}.
In addition, assume that $V$ is a right $H$-comodule nonlocal vertex algebra whose right $H$-comodule structure
$\rho: V\rightarrow V\ot H$ is compatible with $Y_V^H$ and is a $G$-module homomorphism.
Then $\mathfrak D_{Y_V^H}^\rho (V)$ is a $(G,\chi)$-module nonlocal vertex algebra
with the same representation $R$ of $G$ on $V$.
Furthermore, the map $\rho: V\rightarrow V\ot H$ is a homomorphism of  $(G,\chi)$-module nonlocal vertex algebras
from $\mathfrak D_{Y_V^H}^\rho (V)$ to $V\sharp H$.
\end{prop}
%

\begin{rem}
{\em  Let $H$ be a $(G,\chi)$-module nonlocal vertex bialgebra and let $V$ be a $(G,\chi)$-module nonlocal vertex algebra.
As a convention,
we shall always assume that the condition \eqref{RV-RH-compatibility} holds
for an $H$-module nonlocal vertex algebra structure $Y_V^H$ on $V$ and that
 $\rho$ is also a $G$-module homomorphism for an $H$-comodule nonlocal vertex algebra structure $\rho$.}
\end{rem}



Recall that $\C_*((x_1,x_2))$ denotes the fraction field of $\C((x_1,x_2))$ and that
for $F(x_1,x_2)\in \C_*((x_1,x_2))$,  $F(\phi(x,z),x)$ was defined as an element of $\C((x))((z))$.
Here, we interpret this definition in a different way.
For $f(x_1,x_2)\in \C((x_1,x_2))$, we alternatively define $f(\phi(x,z),x)$ in two steps via substitutions
$$\C((x_1,x_2))\overset{s_{x_1=\phi(x,z)}}{\longrightarrow} \C((x,x_2))[[z]]\overset{s_{x_2=x}}{\longrightarrow} \C((x))[[z]]
\subset \C((x))((z)),$$
where  $s_{x_1=\phi(x,z)}$ and $s_{x_2=x}$ denote the indicated substitution maps.
The composition of the two substitutions was proved to be one-to-one.
But, note that the substitution $\C((x,x_2))[[z]]\overset{s_{x_2=x}}{\longrightarrow} \C((x))[[z]]$ is {\em not} one-to-one.
  Set
\begin{align}
\C((x,x_2))[[z]]^{o}=\{ f(x,x_2,z)\in \C((x,x_2))[[z]]\ |\  f(x,x,z)\ne 0\}.
\end{align}
Note that $f(\phi(x,z),x_2)\in \C((x,x_2))[[z]]^{o}$ for $f(x_1,x_2)\in \C((x_1,x_2))^{\times}$.  Then set
\begin{align}
\Lambda(x,x_2,z)=\left\{ \frac{f}{g}\ |\ f\in \C((x,x_2))[[z]],\ g\in \C((x,x_2))[[z]]^{o}\right\}.
\end{align}
The substitution map $s_{x_2=x}$ extends uniquely to a ring embedding
\begin{align}
\tilde{s}_{x_2=x}: \ \Lambda(x,x_2,z)\longrightarrow \C((x))((z)).
\end{align}
Now, for any $F(x_1,x_2)\in \C_{*}((x_1,x_2))$, we have $F(\phi(x,z),x_2)\in  \Lambda(x,x_2,z)$
and
\begin{align}
F(\phi(x,z),x_2)|_{x_2=x}:=\tilde{s}_{x_2=x}(F(\phi(x,z),x_2))=F(\phi(x,z),x).
\end{align}

We have the following result:

\begin{lem}\label{F-characterization}
Let $\phi(x,z)=e^{zp(x)\frac{d}{dx}}x$ with $p(x)\in \C((x))^{\times}$ and let $F(x_1,x_2)\in \C_{*}((x_1,x_2))$. Then
\begin{align}\label{eF-phi-z}
F(\phi(x_1,z),x_2)=F(x_1,\phi(x_2,-z))
\end{align}
holds in $\C_{*}((x_1,x_2))[[z]]$ if and only if
\begin{align}\label{p(x)-differential}
p(x_1)\frac{\partial}{\partial x_1}F(x_1,x_2)=-p(x_2)\frac{\partial}{\partial x_2}F(x_1,x_2).
\end{align}
Furthermore, assuming either one of the two equivalent conditions, we have
\begin{align}
&F(\phi(x,z),x)=f(z),\\
&\iota_{x,z_1,z_2}F(\phi(x,z_1),\phi(x,z_2))=f(z_1-z_2)
\end{align}
for some $f(z)\in \C((z))$.
\end{lem}

\begin{proof} From definition we have
$$F(\phi(x_1,z),x_2)=e^{zp(x_1)\partial_{x_1}}F(x_1,x_2)\
\mbox{ and }\  F(x_1,\phi(x_2,-z))=e^{-zp(x_2)\partial_{x_2}}F(x_1,x_2).$$
This immediately confirms the first assertion. For the second assertion, assuming
$F(\phi(x_1,z),x_2)=F(x_1,\phi(x_2,-z))$,  we have
\begin{align*}
p(x)\frac{\partial}{\partial x}F(\phi(x,z),x)&
=\left(p(x)\frac{\partial}{\partial x}F(\phi(x,z),x_2)+p(x_2)\frac{\partial}{\partial x_2}F(\phi(x,z),x_2)\right)|_{x_2=x}\nonumber\\
&=\left(\frac{\partial}{\partial z}F(\phi(x,z),x_2)+p(x_2)\frac{\partial}{\partial x_2}F(x,\phi(x_2,-z))\right)|_{x_2=x}\nonumber\\
&=\left(\frac{\partial}{\partial z}F(\phi(x,z),x_2)-\frac{\partial}{\partial z}F(x,\phi(x_2,-z))\right)|_{x_2=x}\nonumber\\
&=\frac{\partial}{\partial z}\left(F(\phi(x,z),x_2)-F(x,\phi(x_2,-z))\right)|_{x_2=x}\nonumber\\
&=0.
\end{align*}
As $p(x)\ne 0$, we get $\frac{\partial}{\partial x}F(\phi(x,z),x)=0$. Thus
$F(\phi(x,z),x)\in \C((z))$. Setting $f(z)=F(\phi(x,z),x)$, using (\ref{eF-phi-z}) we have
\begin{align*}
F(\phi(x,z_1),\phi(x,z_2))=F(\phi(\phi(x,z_1),-z_2),x)=F(\phi(x,z_1-z_2),x)=f(z_1-z_2),
\end{align*}
where we are using obvious expansion conventions.
This completes the proof.
\end{proof}

\begin{de}
Let $\phi(x,z)=e^{zp(x)\frac{d}{dx}}x$ with $p(x)\in \C((x))^{\times}$. Denote by $\C_{\phi}((x_1,x_2))$ the set of
all $F(x_1,x_2)\in \C_{*}((x_1,x_2))$ such that (\ref{p(x)-differential}) holds.
\end{de}

Note that with $\phi(x,z)=e^{zp(x)\frac{d}{dx}}x$,  we have
 \begin{align}\label{dz-phi(x,z)-formula}
 \partial_{z}\phi(x,z)=\partial_{z}e^{zp(x)\frac{d}{dx}}x=e^{zp(x)\frac{d}{dx}}p(x)=p(\phi(x,z)).
\end{align}
For $F(x_1,x_2)\in \C_{\phi}((x_1,x_2))$, we get
\begin{eqnarray*}
\partial_zF(\phi(x_2,z),x_2)&=&\left(\partial_{x_1}F(x_1,x_2)\right)|_{x_1=\phi(x_2,z)}\partial_z\phi(x_2,z)\\
&=&\left(p(x_1)\partial_{x_1}F(x_1,x_2)\right)|_{x_1=\phi(x_2,z)}.
\end{eqnarray*}
Then by a straightforward argument we have:

\begin{lem}\label{diff-algebra-embeding}
The set $\C_{\phi}((x_1,x_2))$ is a subalgebra of $\C_{*}((x_1,x_2))$ with $p(x_1)\partial_{x_1}$ as a derivation
and  the map $\pi_{\phi}: \C_{\phi}((x_1,x_2))\rightarrow \C((z))$
defined by
\begin{align}
\pi_{\phi}(F(x_1,x_2))=F(\phi(x,z),x)
\end{align}
for $F(x_1,x_2)\in \C_{\phi}((x_1,x_2))$ is an embedding of differential algebras, where $\C((z))$
is viewed as a differential algebra with derivation $\frac{d}{dz}$.
\end{lem}

With Lemma \ref{diff-algebra-embeding} we define a differential subalgebra of $\C((z))$
\begin{align}
\C_{\phi}((z))=\pi_{\phi}( \C_{\phi}((x_1,x_2)))\subset \C((z)).
\end{align}

We also have the following result, whose proof is straightforward:

\begin{lem}\label{F-f-phi-G}
  Let $F(x_1,x_2)\in \C_{\phi}((x_1,x_2)),\ \lambda\in \C^{\times}$ such that $p(\lambda x)=\mu p(x)$
for some $\mu\in \C^{\times}$. Then $F(\lambda x_1,\lambda x_2)\in \C_{\phi}((x_1,x_2))$ with
\begin{align}
\pi_{\phi}(F(\lambda x_1,\lambda x_2))(z)=\pi_{\phi}(F(x_1,x_2))(\lambda\mu^{-1}z).
\end{align}
\end{lem}
%

\begin{rem}
{\em Let $G$ be a group with two linear characters $\chi$ and $\chi_{\phi}$ such that
$$p(\chi_{\phi}(g)x)=\chi(g)^{-1}\chi_{\phi}(g)p(x)\quad  \te{ for }g\in G$$
(see (\ref{chi-phi-p(x)})).
Assume $F(x_1,x_2)\in \C_{\phi}((x_1,x_2))$. By Lemma \ref{F-f-phi-G} we have
\begin{align}
\pi_{\phi}(F(\chi_{\phi}(g) x_1, \chi_{\phi}(g)x_2))(z)=\pi_{\phi}(F(x_1,x_2))(\chi(g)z).
\end{align}
Let $G$ act on $\C[[x,x^{-1}]]$ by
\begin{align}
(\sigma f)(x)=f(\chi(\sigma)^{-1}x)\quad \te{for }\sigma \in G,\ f(x)\in \C[[x,x^{-1}]].
\end{align}
On the other hand, let $G$ act on $\C[[x_1^{\pm 1},x_2^{\pm 1}]]$ by
\begin{align}
(\sigma F)(x_1,x_2)=F(\chi_{\phi}(\sigma)^{-1}x_1,\chi_{\phi}(\sigma)^{-1}x_2)
\end{align}
for $\sigma \in G,\ F\in \C[[x_1^{\pm 1},x_2^{\pm 1}]]$.
It is straightforward to see that $\C_{\phi}((x))$ and $C_{\phi}((x_1,x_2))$ are $G$-submodules of $\C[[x^{\pm 1}]]$
and $\C[[x_1^{\pm 1},x_2^{\pm 1}]]$, respectively,
and $\pi_{\phi}$ is a $G$-module isomorphism from $\C_{\phi}((x_1,x_2))$ onto $\C_{\phi}((x))$.
For $f(x)\in \C_{\phi}((x))$, set
\begin{align}
\wh f=\pi_{\phi}^{-1}(f)\in \C_{\phi}((x_1,x_2)).
\end{align}
Then $\wh{\sigma f}=\sigma \wh f$ for $f\in \C_{\phi}((x)),\ \sigma\in G$.}
\end{rem}

\begin{rem}
{\em Let $r\in \Z$. Recall (\ref{expression-phi-r}):
\begin{eqnarray*}
\phi_r(x,z)=e^{zx^{r+1}\frac{d}{dx}}x=\begin{cases}x\(1-rzx^{r}\)^{\frac{-1}{r}}&\  \  \mbox{ if }r\ne 0\\
xe^z& \  \  \mbox{ if }r= 0.
\end{cases}
\end{eqnarray*}
Set
\begin{eqnarray}
F_r(z_1,z_2)=\begin{cases}-\frac{1}{r}(z_1^{-r}-z_2^{-r})&\  \  \mbox{ if }r\ne 0\\
z_1/z_2&\  \  \mbox{ if }r= 0,
\end{cases}
\end{eqnarray}
which are elements of  $\C[z_1^{\pm 1},z_2^{\pm 1}]$,
and set
\begin{eqnarray}
f_r(z)=\begin{cases}z&\  \  \mbox{ if }r\ne 0\\
e^{z}& \  \  \mbox{ if }r= 0.
\end{cases}
\end{eqnarray}
It can be readily seen that $F_r(x_1,x_2)\in \C_{\phi_r}((x_1,x_2))$ with $\pi_{\phi_r}(F_r)=f_r$. }
\end{rem}

\begin{rem}\label{r=0-case}
{\em Consider the special case $\phi(x,z)=e^{zx\frac{d}{dx}}x=xe^z$.
It is straightforward to show that for  $F(x_1,x_2)\in \C[[x_1^{\pm 1},x_2^{\pm 1}]]$,
$$x_1\frac{\partial}{\partial x_1}F(x_1,x_2)=-x_2\frac{\partial}{\partial x_2}F(x_1,x_2)$$
if and only if
$F(x_1,x_2)=f(x_1/x_2)$ for some $f(z)\in \C[[z,z^{-1}]]$.
From this we get
\begin{align}
\C((x_1,x_2))\cap \C_{\phi}((x_1,x_2))=\C[x_{12}^{\pm 1}]
\end{align}
(the Laurent polynomial ring in $x_{12}$), where $x_{12}={x_1}/{x_2}$.
It is clear that
\begin{align}
\C(x_1/x_2)\subset \C_{\phi}((x_1,x_2)),
\end{align}
 where $\C(x_1/x_2)$ denotes the field of rational functions in $x_1/x_2$.
Notice that $\C(x_1/x_2)=\C(x_2/x_1)$ inside $\C(x_1,x_2)$ $(\subset \C_{*}((x_1,x_2)))$.}
\end{rem}

Note that in case $\phi(x,z)=e^{z\frac{d}{dx}}x=x+z$, we have
\begin{align}
\C_{\phi}((x_1,x_2))\supset \C((x_1-x_2)), \  \   \C((x_2-x_1)),\   \   \C(x_1-x_2).
\end{align}

For any subgroup $\Gamma$ of $\C^\times$,  set
\begin{align}
\C_{*}^{\Gamma}((x_1,x_2))=\left\{ \frac{F(x_1,x_2)}{q(x_2/x_1)}\mid F(x_1,x_2)\in \C((x_1,x_2)),\ q(x)\in \C_{\Gamma}[x]\right\},
\end{align}
recalling that $\C_{\Gamma}[x]$ denotes the monoid generated multiplicatively by polynomials $x-\alpha$ for $\alpha\in \Gamma$.
Furthermore, set
\begin{align}
\C^{\Gamma}_{\phi}((x_1,x_2))=\C_{\phi}((x_1,x_2))\cap \C^{\Gamma}_{*}((x_1,x_2)).
\end{align}
It is straightforward to see that $\C^{\Gamma}_{\phi}((x_1,x_2))$ is a differential subalgebra of $\C_{\phi}((x_1,x_2))$.
Consequently, $\pi_{\phi}\left(\C^{\Gamma}_{\phi}((x_1,x_2))\right)$ is a differential subalgebra of $\C_{\phi}((x))$.

\begin{de}\label{phi-compatible}
Let $H$ be a nonlocal vertex bialgebra.
A {\em $\phi$-compatible $H$-module nonlocal vertex algebra} is
an $H$-module nonlocal vertex algebra $(V,Y_V^H)$ such that
\begin{align}\label{G-phi-compatible0-new}
  Y_V^H(h,z)v\in V\ot \C_{\phi}((z))\quad\te{for }h\in H,\,v\in V.
\end{align}
If $H$ is a $(G,\chi)$-module nonlocal vertex bialgebra,
we in addition assume that
for $h\in H,\ v\in V$, there exists $q(x)\in \C_{\chi_\phi(G)}[x]$ such that
\begin{align}
q(x_2/x_1)\wh Y_V^H(h,x_1,x_2)v\in &V\ot \C((x_1,x_2)),
\label{G-phi-compatible}
\end{align}
where $\wh Y_V^H(h,x_1,x_2)v=(1\ot \pi_{\phi}^{-1})Y_V^H(h,x)v$.
\end{de}

Note that we can alternatively combine the conditions (\ref{G-phi-compatible0-new}) and (\ref{G-phi-compatible}) to write
\begin{align}\label{G-phi-compatible0-stronger}
  Y_V^H(h,z)v\in V\ot \pi_{\phi}(\C_{\phi}^{\chi_\phi(G)}((x_1,x_2)))\quad \te{ for }h\in H,\ v\in V.
\end{align}
For convenience, we formulate the following technical result:

\begin{lem}\label{phi-compatible-va-tech}
Let $(V,Y_V^H)$ be an $H$-module nonlocal vertex algebra. Assume that (\ref{G-phi-compatible0-new}) holds for $h\in S,\ v\in U$,
where $S$ is a generating subspace of $H$ as a nonlocal vertex algebra such that $\Delta(S)\subset S\otimes S$
and $U$ is a generating subset of $V$. Then  $(V,Y_V^H)$ is $\phi$-compatible.
\end{lem}

\begin{proof} Set $V'=\{v\in V\mid Y_V^H(h,x)v\in V\ot \C_{\phi}((x))\  \te{for }h\in S\}$. Then $U\cup\{{\bf 1}\}\subset V'$.
Let $h\in S,\ u,v\in V'$. With $V$ an $H$-module nonlocal vertex algebra, we have
$$Y_{V}^H(h,x)Y(u,z)v=\sum Y(Y_V^H(h_{(1)},x-z)u,z)Y_V^H(h_{(2)},x)v.$$
Writing
$$Y_V^H(h_{(1)},x)u=\sum_{i=1}^r a^{(i)}\ot f_i(x),\quad  Y_V^H(h_{(2)},x)v=\sum_{j=1}^s b^{(j)}\ot g_i(x),$$
where $a^{(i)},b^{(j)}\in V,\ f_i(x),g_j(x)\in  \C_{\phi}((x))$, we have
\begin{align*}
Y_{V}^H(h,x)Y(u,z)v=\sum_{i,j}g_j(x)e^{-z\frac{\partial}{\partial x}}f_i(x)Y(a^{(i)},z)b^{(j)}.
\end{align*}
From this we get $Y_{V}^H(h,x)u_mv\in V\ot \C_{\phi}((x))$ for all $m\in \Z$,
recalling that $\C_{\phi}((x))$ is a subalgebra of $\C((x))$, which is closed under the derivation $\frac{d}{dx}$.
As $U$ generates $V$, it follows that $V'=V$. Thus $Y_{V}^H(h,x)v\in V\ot \C_{\phi}((x))$ for all $h\in S,\ v\in V$.

Furthermore, for $h,h'\in H,\ v\in V$, by Lemma \ref{lem:strong-asso} we have
$$Y_V^H(Y(h,z)h',x)v=Y_V^H(h,x+z)Y_V^H(h',x)v.$$
Using a similar argument we get $Y_V^H(h_mh',x)v\in V\ot \C_{\phi}((x))$ for all $m\in \Z$.
It follows that (\ref{G-phi-compatible0-new}) holds for all $h\in H,\ v\in V$.  Therefore, $(V,Y_V^H)$ is $\phi$-compatible.
\end{proof}

Let $(V,Y_V^H)$ be a $\phi$-compatible $H$-module nonlocal vertex algebra.
From definition, we have
\begin{align}\label{hat-expression}
\wh Y_V^H(h,x_1,x_2)v\in V\ot \C_{\phi}((x_1,x_2))\quad \te{for }h\in H,\ v\in V.
\end{align}
The following are immediate consequences:

\begin{lem}\label{lem:phi-compatible-properties}
Let $(V,Y_V^H)$ be a $\phi$-compatible $H$-module nonlocal vertex algebra. For $h,h'\in H,\ v\in V$, we have
\begin{align}
&\wh Y_V^H(h,\phi(x_1,z),x_2)=\wh Y_V^H(h,x_1,\phi(x_2,-z)),\label{phi-compatible1}\\
&\wh Y_V^H\(Y(h,z)h',x_1,x_2\)=\wh Y_V^H(h,\phi(x_1,z),x_2)\wh Y_V^H(h',x_1,x_2),
\label{phi-compatible3}\\
&\wh Y_V^H(h,x_1,x_2)Y(v,x_3)=\sum Y(\wh Y_V^H(h_{(1)},x_1,\phi(x_2,x_3))v,x_3)
\wh Y_V^H(h_{(2)},x_1,x_2).\label{phi-compatible4}
\end{align}
Furthermore, we have
\begin{align}
g(\phi(x,z),x)\wh Y_V^H(h,\phi(x,z),x)v=g(\phi(x,z),x)Y_V^H(h,z)v\label{connection-1}
\end{align}
for any $g(x_1,x_2)\in\C((x_1,x_2))$ such that
$$g(x_1,x_2)\wh Y_V^H(h,x_1,x_2)v\in V\ot\C((x_1,x_2)).$$
\end{lem}

\begin{de}\label{de:G-phi-compatible-mod}
Let $H$ be a $(G,\chi)$-module nonlocal vertex bialgebra,
let $(V,Y_V^{H})$ be a $\phi$-compatible $H$-module nonlocal vertex algebra,
and let $\chi_\phi$ be a linear character of $G$ satisfying (\ref{p(x)-compatibility}).
A {\em $(G,\chi_{\phi})$-equivariant $\phi$-coordinated quasi $(H,V)$-module}
is a $(G,\chi_{\phi})$-equivariant $\phi$-coordinated quasi $H$-module $(W,Y_W^H)$, equipped with a
$(G,\chi_{\phi})$-equivariant $\phi$-coordinated quasi $V$-module structure $Y_W^V(\cdot,x)$,
such that
\begin{align}
&Y_W^H(h,x)w\in W\ot\C((x)),
    \label{eq:G-phi-compatible7}\\
&Y_W^H(h,x_1)Y_W^V(v,x_2)=\sum Y_W^V\( \iota_{x_2,x_1}\wh{Y}_V^H(h_{(1)},
    x_1,x_2)v,x_2\)Y_W^H(h_{(2)},x_1)
    \label{eq:G-phi-compatible8}
\end{align}
 for  $h\in H$, $v\in V,$ $w\in W$.
\end{de}

The following lemma follows  immediately from the proof of Lemma \ref{lem:strong-asso}:

\begin{lem}\label{strong-assoc-phi}
Let $(W,Y_W^H)$ be a $\phi$-coordinated quasi $H$-module such that
(\ref{eq:G-phi-compatible7}) holds for any $h\in H,\ w\in W$.   Then
\begin{align}
Y_W^H(h,x_1)Y_W^H(k,x_2)|_{x_1=\phi(x_2,z)}=Y_W^H(Y(h,z)k,x_2)
\end{align}
for $h,k\in H$.
\end{lem}

As the main result of this section, we have (cf. Proposition \ref{prop:smash-mod-1}):

\begin{thm}\label{prop:deform-phi-quasi-mod}
Let $H$ be a $(G,\chi)$-module nonlocal vertex bialgebra,
 let $(V,Y_V^H)$ be a $\phi$-compatible $H$-module nonlocal vertex algebra,
and let $(W,Y_W^H,Y_W^V)$ be a $(G,\chi_{\phi})$-equivariant $\phi$-coordinated quasi $(H,V)$-module.
In addition, assume
\begin{eqnarray*}
R_{V}(g)Y_V^H(h,x)v=Y_V^H(R_{H}(g) h,\chi(g)x)R_{V}(g) v\   \   \   \mbox{ for }g\in G,\ h\in H,\ v\in V
\end{eqnarray*}
(see Proposition \ref{smash-G-algebra}).
For $h\in H,\  v\in V$, set
\begin{eqnarray}
  Y_W^{\sharp}(v\ot h,x)=Y_W^V(v,x)Y_W^H(h,x)\in ({\rm End} W)[[x,x^{-1}]].
\end{eqnarray}
Then $(W,Y_W^{\sharp})$ is a $(G,\chi_{\phi})$-equivariant $\phi$-coordinated quasi $V\sharp H$-module.
\end{thm}

\begin{proof}  First of all, for $v\in V,\ h\in H$,  with the property (\ref{eq:G-phi-compatible7}) we have
$$Y_W^{\sharp}(v\ot h,x)=Y_W^V(v,x)Y_W^H(h,x)\in \te{Hom}(W,W((x))).$$
We also have
$$Y_W^{\sharp}({\bf 1}\ot {\bf 1},x)=Y_W^V({\bf 1},x)Y_W^H({\bf 1},x)=1.$$
Let $u,v\in V$, $h,k\in H$. From definition,  there exists $q(x)\in\C_{\chi(G)}[x]$ such that
\begin{align*}
 &q(x_1/x_2)\wh Y_V^H\(h_{(1)},x_1,x_2\)v \in V\ot \C((x_1,x_2)),\\
 & q(x_1/x_2)Y_W^V(u,x_1)
    Y_W^V\(\wh Y_V^H(h_{(1)},x_1,x_2)v,x_2\)\in \te{Hom} (W,W((x_1,x_2))).
\end{align*}
On the other hand,  from (\ref{eq:G-phi-compatible7}) we have
$$ Y_W^H(h_{(2)},x_1)Y_W^H(k,x_2)\in \te{Hom}(W,W\otimes \C((x_1,x_2))).  $$
Then
\begin{align*}
&q(x_1/x_2)Y_W^{\sharp}(u\ot h,x_1)Y_W^{\sharp}(v\ot k,x_2)\\
=\ &q(x_1/x_2)\sum Y_W^V(u,x_1)Y_W^V\(\iota_{x_2,x_1}\wh Y_V^H(h_{(1)},x_1,x_2)v,x_2\)Y_W^H(h_{(2)},x_1)Y_W^H(k,x_2),
\end{align*}
which implies that the common quantity on both sides lies in $\te{Hom}(W,W((x_1,x_2)))$.
Using all of these, Lemma \ref{strong-assoc-phi}, and (\ref{connection-1}), we get
\begin{align*}
&\(q(x_1/x_2)Y_W^{\sharp}(u\ot h,x_1)Y_W^{\sharp}(v\ot k,x_2)\)|_{x_1=\phi(x_2,z)}\\
  =\ &\Big(q(x_1/x_2)\sum Y_W^V(u,x_1)Y_W^V\(\wh Y_V^H(h_{(1)},x_1,x_2)v,x_2\)
  Y_W^H(h_{(2)},x_1)Y_W^H(k,x_2)\Big)|_{x_1=\phi(x_2,z)}\\
 =\  &q(\phi(x_2,z)/x_2)Y_W^V\(Y(u,z)\wh Y_V^H(h_{(1)},\phi(x_2,z),x_2)v,x_2\)Y_W^H(Y(h_{(2)},z)k,x_2)\\
 =\ &q(\phi(x_2,z)/x_2)Y_W^V(Y(u,z)Y_V^H(h_{(1)},z)v,x_2)Y_W^H(Y(h_{(2)},z)k,x_2)\\
 =\  &q(\phi(x_2,z)/x_2)Y_W^{\sharp}(Y^\sharp(u\ot h,z)(v\ot k),x_2).
\end{align*}
This proves that $(W,Y_W^{\sharp})$ carries the structure of a $\phi$-coordinated quasi $V\sharp H$-module.

Furthermore, for $g\in G,\ v\in V,\ h\in H$, we have
\begin{align*}
Y_W^{\sharp}&\((R_V(g)\ot R_H(g))(v\ot h),x\)=\ Y_{W}^V(R_V(g)v,x)Y_{W}^{H}(R_H(g)h,x)\\
&=\ Y_{W}^V(v,\chi_{\phi}(g)^{-1}x)Y_{W}^{H}(h,\chi_{\phi}(g)^{-1}x)
=Y_W^{\sharp}(u\ot h,\chi_{\phi}(g)^{-1}x).
\end{align*}
This proves that $(W,Y_W^{\sharp})$ is $(G,\chi_{\phi})$-equivariant.
\end{proof}

Recall from Theorem  \ref{prop:deform-va} that the right $H$-comodule map $\rho: V \rightarrow V\ot H$ is
also a nonlocal vertex algebra homomorphism from $\mathfrak D_{Y_V^H}^\rho (V)$ to $V\sharp H$.
Note that $\mathfrak D_{Y_V^H}^\rho (V)$ is a  $(G,\chi)$-module nonlocal vertex algebra.
As an immediate consequence of Theorem \ref{prop:deform-phi-quasi-mod} we have:

\begin{coro}
Under the setting of Theorem \ref{prop:deform-phi-quasi-mod}, for $v\in V$, set
\begin{align}
Y_{W}^{\sharp \rho}(v,x)=\sum Y_{W}^V(v_{(1)},x)Y_{W}^H(v_{(2)},x)
\  \left(=Y_W^{\sharp}(\rho(v),x)\right).
\end{align}
In addition we assume that $H$ is cocommutative.
Then $(W,Y_W^{\sharp\rho})$ is a $(G,\chi_{\phi})$-equivariant $\phi$-coordinated quasi
$\mathfrak D_{Y_V^H}^\rho (V)$-module.
\end{coro}
%

For any subgroup $\Gamma$ of $\C^{\times}$, set
\begin{align}\label{Gamma-polynomial-monoid}
\C_{\Gamma}[x_1,x_2]=\langle x_1-\alpha x_2\ |\ \alpha\in \Gamma\rangle,
\end{align}
the multiplicative moniod generated by $x_1-\alpha x_2$ for $\alpha\in \Gamma$ in $\C[x_1,x_2]$.

\begin{lem}\label{lem:converse0}
Let $V$ be a $(G,\chi)$-module nonlocal vertex algebra and
let $(W,Y_W^\phi)$ be a $(G,\chi_{\phi})$-equivariant $\phi$-coordinated quasi $V$-module.
Suppose that $u,v\in V$, and
\begin{align*}
  u^{(i)},v^{(i)}\in V,\   f_{i}(x)\in\C_{\phi}((x))\  \  (i=1,\dots,r)
\end{align*}
such that for some nonnegative integer $k$,
\begin{align}\label{eq:converse0-va-rel}
 &(x_1-x_2)^{k}Y(u,x_1)Y(v,x_2)\nonumber\\
 =&(x_1-x_2)^k\sum_{i=1}^r
f_i(-x_2+x_1)Y(v^{(i)},x_2)Y(u^{(i)},x_1).
\end{align}
Then there exists $q(x_1,x_2)\in \C_{\chi_{\phi}(G)}[x_1,x_2]$
such that
\begin{align}\label{q-phi-module-Slocality}
&  q(x_1,x_2)Y_{W}^\phi(u,x_1)Y_{W}^\phi(v,x_2)\nonumber\\
=\ &q(x_1,x_2)\sum_{i=1}^r\iota_{x_2,x_1}
\(\wh{f_i}(x_1,x_2)\)
  Y_{W}^\phi(v^{(i)},x_2)Y_{W}^\phi(u^{(i)},x_1).
\end{align}
 Furthermore, (\ref{q-phi-module-Slocality}) holds for any $q(x_1,x_2)\in\C((x_1,x_2))$ such that
\begin{align*}
  q(x_1,x_2)Y_{W}^\phi(u,x_1)Y_{W}^\phi(v,x_2)\in\te{Hom}(W,W((x_1,x_2))).
\end{align*}
\end{lem}

\begin{proof} With \eqref{eq:converse0-va-rel}, from \cite[Corollary 5.3]{li-nonlocal} we have
\begin{align*}
  Y(u,x)v=\sum_{i=1}^rf_i(x)e^{x\der}Y(v^{(i)},-x)u^{(i)}.
\end{align*}
From definition, there exists $h(x_1,x_2)\in\C_{\chi_{\phi}(G)}[x_1,x_2]$ such that
\begin{align*}
  h(x_1,x_2)Y_W^\phi(u,x_1)Y_W^\phi(v,x_2)\in\te{Hom}(W,W((x_1,x_2))),
\end{align*}
 $h(x_1,x_2)\wh f_i(x_1,x_2)\in \C((x_1,x_2))$, and
\begin{align*}
  (h(x_1,x_2)\wh f_i(x_1,x_2))Y_W^\phi(v^{(i)},x_2)Y_W(u^{(i)},x_1)\in\te{Hom}(W,W((x_1,x_2)))
\end{align*}
for $i=1,\dots,r$.
Then using \cite[Lemma 3.7]{li-cmp} we get
\begin{align*}
  &(h(x_1,x_2)Y_W^\phi(u,x_1)Y_W^\phi(v,x_2))|_{x_1=\phi(x_2,z)}\\
  =\  &h(\phi(x_2,z),x_2)Y_W^\phi(Y(u,z)v,x_2)\\
  =\  &\sum_{i=1}^r(h(\phi(x_2,z),x_2)f_i(z)
    Y_W^\phi(e^{z\der}(v^{(i)},-z)u^{(i)},x_2)\\
  =\  &\sum_{i=1}^r(h(\phi(x_2,z),x_2)f_i(z)
    Y_W^\phi(Y(v^{(i)},-z)u^{(i)},\phi(x_2,z)).
\end{align*}
On the other hand, we have
\begin{align*}
  &\(h(x_1,x_2)\sum_{i=1}^r\iota_{x_2,x_1}(\wh f_i(x_1,x_2))Y_W^\phi(v^{(i)},x_2)Y_W^\phi(u^{(i)},x_1)\)|_{x_2=\phi(x_1,-z)}\\
  =&\sum_{i=1}^rh(x_1,\phi(x_1,-z))\wh f_i(x_1,\phi(x_1,-z))Y_W^\phi(Y(v^{(i)},-z)u^{(i)},x_1)\\
 = &\sum_{i=1}^rh(x_1,\phi(x_1,-z)) f_i(z)Y_W^\phi(Y(v^{(i)},-z)u^{(i)},x_1).
\end{align*}
Then using \cite[Remark 2.8]{li-cmp} we have
\begin{align*}
 & (h(x_1,x_2)Y_W^\phi(u,x_1)Y_W^\phi(v,x_2))|_{x_1=\phi(x_2,z)}\\
  =&\(\(h(x_1,x_2)\sum_{i=1}^r\wh f_i(x_1,x_2)Y_W^\phi(v^{(i)},x_2)Y_W^\phi(u^{(i)},x_1)\)|_{x_2=\phi(x_1,-z)}\)
  |_{x_1=\phi(x_2,z)}\\
 =&\(h(x_1,x_2)\sum_{i=1}^r\iota_{x_2,x_1}(\wh f_i(x_1,x_2))Y_W^\phi(v^{(i)},x_2)Y_W^\phi(u^{(i)},x_1)\)|_{x_1=\phi(x_2,z)},
\end{align*}
noticing that $\phi(\phi(x_1,z),-z)=\phi(x_1,0)=x_1$.
By \cite[Remark 2.8]{li-cmp} again we get
\begin{align*}
  h(x_1,x_2)Y_W^\phi(u,x_1)Y_W^\phi(v,x_2)
  =h(x_1,x_2)\sum_{i=1}^r\iota_{x_2,x_1}(\wh f_i(x_1,x_2))Y_W^\phi(v^{(i)},x_2)Y_W^\phi(u^{(i)},x_1).
\end{align*}
This proves the first assertion. As for the second assertion we have
\begin{align*}
  &h(x_1,x_2)\left[q(x_1,x_2)Y_W^\phi(u,x_1)Y_W^\phi(v,x_2)\right]\\
  =\  &h(x_1,x_2)q(x_1,x_2)\sum_{i=1}^r\wh f_i(x_1,x_2)Y_W^\phi(v^{(i)},x_2)Y_W^\phi(u^{(i)},x_1).
\end{align*}
Multiplying both sides by the inverse of $h(x_1,x_2)$ in $\C((x_2))((x_1))$ we obtain the desired relation.
\end{proof}

\begin{rem}\label{remark-basechange-fact}
{\em  Let $p(x)\in \C((x))^{\times}$.
It follows from induction that for every positive integer $n$, there exist $A_{n,0}(x),\dots,A_{n,n}(x)\in \C((x))$
with $A_{n,n}(x)=p(x)^{-n}$ such that
\begin{align}\label{base-change}
\left(\frac{d}{dx}\right)^n=A_{n,0}(x)+A_{n,1}(x)p(x)\frac{d}{dx}+\cdots +A_{n,n}(x)\left(p(x)\frac{d}{dx}\right)^n
\end{align}
on $\C((x))$. For $B_{n,0}(x),\dots, B_{n,n}(x)\in \C((x))$, it can be readily seen that
$$B_{n,0}(x)+B_{n,1}(x)\frac{d}{dx}+\cdots +B_{n,n}(x)\left(\frac{d}{dx}\right)^n=0$$
if and only if $B_{n,i}(x)=0$ for $i=0,\dots,n$. Then it follows that those coefficients $A_{n,i}(x)$ for $0\le i\le n$ in (\ref{base-change})
 are uniquely determined.}
\end{rem}

The following result generalizes the corresponding results in \cite{li-nonlocal} and \cite{li-jmp}:

\begin{prop}\label{prop:reverse-cal}
Let $V$ be a $(G,\chi)$-module nonlocal vertex algebra,
$(W,Y_W^\phi)$ a $(G,\chi_{\phi})$-equivariant $\phi$-coordinated quasi $V$-module, and
let $G^0$ be a complete set of coset representatives of $\ker \chi_{\phi}$ in $G$.
Assume $u,v\in V$, and
\begin{align*}
  u^{(i)},v^{(i)}\in V,\   f_{i}(x)\in\C_{\phi}((x))\  \  (i=1,\dots,r)
\end{align*}
such that
\begin{align*}
 (x_1-x_2)^{k}Y(u,x_1)Y(v,x_2)=(x_1-x_2)^k\sum_{i=1}^r
 f_i(-x_2+x_1)Y(v^{(i)},x_2)Y(u^{(i)},x_1)
\end{align*}
for some nonnegative integer $k$. Then
\begin{align}\label{phi-module-commuator}
  &Y_W^\phi(u,x_1)Y_W^\phi(v,x_2)-
  \sum_{i=1}^r\iota_{x_2,x_1}(\widehat{f_i}(x_1,x_2))
  Y_W^\phi(v^{(i)},x_2)Y_W^\phi(u^{(i)},x_1)\nonumber\\
  =&\sum_{g\in G^{0}}\sum_{j\ge0}
  Y_W^\phi\((R(g)u)_jv,x_2\)\frac{1}{j!}  (p(x_2)\partial_{x_2})^j
  p(x_1)x_1\inverse\delta\(\chi_{\phi}(g)^{-1}\frac{x_2}{x_1}\).
\end{align}
\end{prop}

\begin{proof}
From Lemma \ref{lem:converse0}, there exist distinct (nonzero complex numbers) $c_1,\dots,c_s\in \chi_{\phi}(G)$ and
positive integers $k_1,\dots,k_s$ such that  (\ref{q-phi-module-Slocality}) holds with
$$q(x_1,x_2)=(x_1-c_1x_2)^{k_1}\cdots (x_1-c_sx_2)^{k_s}.$$
By \cite[Lemma 2.5]{li-new}, we have
\begin{align*}
&p(x_1)^{-1}Y_W^\phi(u,x_1)Y_W^\phi(v,x_2)-p(x_1)^{-1}
  \sum_{i=1}^r\iota_{x_2,x_1}(\widehat{f_i}(x_1,x_2))
  Y_W^\phi(v^{(i)},x_2)Y_W^\phi(u^{(i)},x_1)\nonumber\\
  =\ & \sum_{j=1}^s\sum_{t=0}^{k_j-1} C_{j,t}(x_2)\left(\frac{\partial}{\partial x_2}\right)^t x_1^{-1}\delta\left(\frac{c_jx_2}{x_1}\right),
\end{align*}
where $C_{j,t}(x)\in \E(W)$. Furthermore, using Remark \ref{remark-basechange-fact}, we get
\begin{align}\label{e4.47}
&Y_W^\phi(u,x_1)Y_W^\phi(v,x_2)-
  \sum_{i=1}^r\iota_{x_2,x_1}(\widehat{f_i}(x_1,x_2))
  Y_W^\phi(v^{(i)},x_2)Y_W^\phi(u^{(i)},x_1)\nonumber\\
  =\ & \sum_{j=1}^s\sum_{t=0}^{k_j-1} A_{j,t}(x_2)\frac{1}{t!}\left(p(x_2)\frac{\partial}{\partial x_2}\right)^t p(x_1)
  x_1^{-1}\delta\left(\frac{c_jx_2}{x_1}\right),
\end{align}
where $A_{j,t}(x)\in \E(W)$. By \cite[Theorems 2.19, 2.21]{JKLT1} and Lemma \ref{nonlocalva-fact} we have
\begin{align}
Y_W^{\phi}(u,c_jx)_{t}^{\phi}Y_W^{\phi}(v,x)=\begin{cases}A_{j,t}(x) &
\quad \te{ for }1\le j\le s,\  0\le t\le k_j-1,\\
0&\quad \te{ for }1\le j\le s,\  t\ge k_j.\end{cases}
\end{align}
Let $g_1,\dots,g_r\in G^0$ such that $c_j=\chi_{\phi}(g_j)^{-1}$ for $1\le j\le s$. Then using \cite[Proposition 4.11]{li-cmp} we get
\begin{align*}
Y_W^{\phi}(u,c_jx)_{t}^{\phi}Y_W^{\phi}(v,x)=Y_W^{\phi}(R(g_j)u,x)_{t}^{\phi}Y_W^{\phi}(v,x)=Y_{W}^{\phi}((R(g_j)u)_tv,x).
\end{align*}
Consequently, we have
\begin{align*}
&Y_W^\phi(u,x_1)Y_W^\phi(v,x_2)-
  \sum_{i=1}^r\iota_{x_2,x_1}(\wh f_{i}(x_1,x_2))
  Y_W^\phi(v^{(i)},x_2)Y_W^\phi(u^{(i)},x_1)\nonumber\\
=\  &\sum_{j=1}^s\sum_{t=0}^{k_j-1}Y_W^\phi((R(g_j)u)_tv,x_2)
\frac{1}{t!}(p(x_2)\partial_{x_2})^tp(x_1)x_1\inverse\delta\(\chi_{\phi}(g_j)^{-1}\frac{x_2}{x_1}\).
\end{align*}
Suppose $g\in G^0$ such that $\chi_{\phi}(g)^{-1}\notin \{ c_1,\dots,c_s\}$.
By \cite[Theorems 2.19, 2.21]{JKLT1} also
(by adding a term with zero coefficients corresponding to $c$ on the right-hand side of (\ref{e4.47}))
we have
$Y_W^{\phi}((R(g)u)_tv,x)=0$ for $t\ge 0$.
Therefore (\ref{phi-module-commuator}) holds. This completes the proof.
\end{proof}

\section{Deformations of lattice vertex algebras $V_L$}

In this section, we use the  vertex bialgebra $B_L$ introduced in \cite{Li-smash} to study the lattice vertex algebra $V_{L}$.
More specifically, we give a right $B_L$-comodule nonlocal vertex algebra structure
and a family of compatible $B_L$-module nonlocal vertex algebra structures on $V_L$.
As an application, we obtain a family of deformations of  $V_L$.

\subsection{Vertex algebra $V_{L}$ and associative algebra $A(L)$}
We first recall the lattice vertex algebra $V_L$ (see  \cite{bor}, \cite{FLM}).
Let $L$ be a finite rank  {\em even lattice,} i.e., a free abelian group of finite rank equipped with a
symmetric $\Z$-valued bilinear form $\<\cdot,\cdot\>$ such that $\<\alpha,\alpha\>\in 2\Z$ for $\alpha\in L$.
We assume that $L$ is nondegenerate in the obvious sense.
Set
\begin{align}
\h=\C\ot_{\Z}L
\end{align}
 and extend $\<\cdot,\cdot\>$ to a symmetric $\C$-valued bilinear form on $\h$.
View $\h$ as an abelian Lie algebra with $\<\cdot,\cdot\>$ as a nondegenerate symmetric invariant bilinear form.
Then we have an affine Lie algebra $\wh \h$, where
$$\wh \h=\h\otimes \C[t,t^{-1}]+\C {\bf k}$$
as a vector space, and where ${\bf k}$ is central and
\begin{eqnarray}
[\alpha(m),\beta(n)]=m\delta_{m+n,0}\<\alpha,\beta\>{\bf k}
\end{eqnarray}
for $\alpha,\beta\in \h,\ m,n\in \Z$ with $\alpha(m)$ denoting $\alpha\otimes t^{m}$.  Set
\begin{eqnarray}
\wh \h^{\pm}=\h\otimes t^{\pm 1} \C[t^{\pm 1}],
\end{eqnarray}
which are abelian Lie subalgebras. Identify $\h$ with $\h\otimes t^{0}$.
Furthermore, set
\begin{eqnarray}
\wh \h'=\wh \h^{+}+\wh \h^{-}+\C {\bf k},
\end{eqnarray}
which is a Heisenberg algebra. Then $\wh \h=\wh \h'\oplus \h$, which is a direct sum of Lie algebras.


Let $\varepsilon:L\times L\rightarrow \C^\times$ be a 2-cocycle such that
\begin{eqnarray*}
\varepsilon(\alpha,\beta)\varepsilon(\beta,\alpha)^{-1}=(-1)^{\<\al,\be\>},\   \   \
\varepsilon(\alpha,0)=1=\varepsilon(0,\alpha)\quad \te{ for } \alpha,\beta \in L.
\end{eqnarray*}
Denote by $\C_{\varepsilon}[L]$ the $\varepsilon$-twisted group algebra of $L$,
which by definition has a designated basis $\{ e_{\alpha}\ |\ \alpha\in L\}$ with relations
\begin{align}
e_\al \cdot e_\be=\varepsilon(\al,\be)e_{\al+\be}\quad \te{ for }\al,\be\in L.
\end{align}
Make $\C_{\varepsilon}[L]$ an $\wh \h$-module by letting $\wh \h'$
act trivially and letting $\h$ $(=\h(0))$ act by
\begin{eqnarray}
  h(0)e_\be=\<h,\be\>e_\be \   \  \mbox{ for }h\in \h,\  \be\in L.
  \end{eqnarray}

Note that $S(\wh \h^{-})$ is naturally an $\widehat \h$-module of level $1$.
Set
\begin{eqnarray}
V_{L}=S(\wh \h^{-})\otimes \C_{\varepsilon}[L],
\end{eqnarray}
the tensor product of $\widehat \h$-modules, which is an $\widehat \h$-module of level $1$.
Set
$$\vac=1\ot e_0\in  V_{L}.$$
Identify $\h$ and $\C_{\varepsilon}[L]$ as subspaces of $V_{L}$ via the correspondence
$$a\mapsto a(-1)\otimes 1\   (a\in\h) \quad\te{and}\quad e_\al\mapsto 1\otimes e_\al\   (\al\in L).$$

For $h\in \h$, set
\begin{align}
h(z)=\sum_{n\in\Z}h(n)z^{-n-1}.
\end{align}
On the other hand, for $\al\in L$ set
\begin{align}
E^{\pm}(\alpha,z)=\exp \left(\sum_{n\in \Z_{+}}\frac{\al(\pm n)}{\pm n}z^{\mp n}\right)
\end{align}
on $V_{L}$.  For $\al\in L$,  define a linear operator
$z^{\al}:\   \C_{\varepsilon}[L]\rightarrow \C_{\varepsilon}[L][z,z^{-1}]$ by
\begin{eqnarray}
z^\al\cdot e_\be=z^{\<\al,\be\>}e_\be\   \   \   \mbox{ for }\be\in L.
\end{eqnarray}
Then  there exists a vertex algebra structure on $V_L$, which is uniquely determined by
the conditions that $\vac$ is the vacuum vector and that $(h\in\h,\,\al\in L)$:
\begin{align}
&Y(h,z)=h(z),\quad
Y(e_\al,z)=E^{-}(-\al,z)E^{+}(-\al,z)e_\al z^\al.
\end{align}


Next, we recall from \cite[Section 6.5]{LL} the associative algebra $A(L)$.
By definition, $A(L)$ is the associative algebra over $\C$ with unit $1$, generated by
$$\set{h[n],\  e_\al[n]}{h\in\h,\ \al\in L,\ n\in \Z},$$
where $h[n]$ is linear in $h$, subject to a set of relations written in terms of generating functions
\begin{align*}
  h[z]=\sum_{n\in\Z}h[n]z^{-n-1},\   \   \  \
  e_\al[z]=\sum_{n\in\Z}e_\al[n]z^{-n-1}.
\end{align*}
The relations are
\begin{align*}
  &\te{(AL1) }\  \   \   \   e_0[z]=1,\\
  &\te{(AL2) }\   \   \   \    \left[h[z_1],
    h'[z_2]\right]=\<h,h'\>\frac{\partial}{\partial z_2}z_1\inverse\delta\(\frac{z_2}{z_1}\),\\
  &\te{(AL3) }\    \   \  \    \left[h[z_1],e_\al[z_2]\right]=\<\al,h\>
    e_\al[z_2]z_1\inverse\delta\(\frac{z_2}{z_1}\),\\
  &\te{(AL4) }\    \   \   \   \left[e_\al[z_1],e_\be[z_2]\right]=0\   \te{ if }\<\al,\be\>\ge0,\\
  &\te{(AL5) }\   \   \   \  (z_1-z_2)^{-\<\al,\be\>}\left[e_\al[z_1],e_\be[z_2]\right]=0\   \te{ if }\<\al,\be\><0,
\end{align*}
for $h, h'\in \h,\  \al,\be\in L$.
An $A(L)$-module $W$ is said to be {\em restricted} if for every $w\in W$, we have
$ h[z]w, \  e_\al[z]w\in W((z))$ for all $h\in\h,\  \al\in L$.

The following result was obtained in  \cite{LL}:

\begin{prop}\label{prop:VQ-mod-AQ-mod}
 For any $V_L$-module $(W,Y_W)$, $W$ is a restricted $A(L)$-module with
\begin{align*}
  h[z]=Y_W(h,z),\   \  \   \,\,e_\al[z]=Y_W(e_\al,z)
\end{align*}
for  $h\in\h,\   \al\in L$, such that the following relations hold on $W$ for $\al,\be\in L$:
\begin{align*}
&\te{(AL6) }\   \   \   \   \partial_z e_\al[z]=\al[z]^{+} e_\al[z]+e_\al[z]\al[z]^{-},\\
  &\te{(AL7) }\   \   \    \   \te{Res}_{x}\big(
    (x-z)^{-\<\al,\be\>-1} e_\al[x] e_\be[z]-(-z+x)^{-\<\al,\be\>-1} e_\be[z] e_\al[x]\big)\\
  &\hspace{1.75cm}
  =\varepsilon(\al,\be) e_{\al+\be}[z],
\end{align*}
where $\al[z]^{+}=\sum_{n<0}\al[n]z^{-n-1}$ and $\al[z]^{-}=\sum_{n\ge 0}\al[n]z^{-n-1}$.
On the other hand, if $W$ is a restricted $A(L)$-module satisfying (AL6) and (AL7),
then $W$ admits a  $V_L$-module structure $Y_{W}(\cdot,z)$ which is uniquely determined by
$$Y_W(h,z)= h[z],\   \   \    Y_W(e_\al,z)= e_\al[z]\   \  \   \mbox{ for }h\in\h,\   \al\in L.$$
\end{prop}

Let $(W,Y_W)$ be any $V_L$-module. From \cite[Proposition 6.5.2]{LL},
the following relations hold on $W$ for $\al,\be\in L$:
\begin{align}\label{YW-ealpha-ebeta}
&(x-z)^{-\<\al,\be\>-1}Y_W(e_{\alpha},x)Y_W(e_{\be},z)
-(-z+x)^{-\<\al,\be\>-1}Y_W(e_{\be},z)Y_W(e_{\alpha},x)\nonumber\\
&\quad \quad \quad = \varepsilon(\al,\be)Y_W(e_{\al+\be},z)x^{-1}\delta\left(\frac{z}{x}\right).
\end{align}
It can be readily seen that this implies  the relations (AL4-5) and (AL7)
with $e_\al[z]=Y_W(e_\al,z)$ for  $\al\in L$. In fact, (\ref{YW-ealpha-ebeta}) is equivalent to (AL7) together with
\begin{align}\label{YW-ealpha-ebeta-locality}
(x-z)^{-\<\al,\be\>}Y_W(e_{\alpha},x)Y_W(e_{\be},z)=
(-z+x)^{-\<\al,\be\>}Y_W(e_{\be},z)Y_W(e_{\alpha},x).
\end{align}

The following is a universal property of the vertex algebra $V_{L}$ (see \cite{JKLT1}):

\begin{prop}\label{prop:AQ-mod-hom-VA-hom}
Let $V$ be a nonlocal vertex algebra
and let $\psi:\h\oplus\C_{\varepsilon}[L]\rightarrow V$ be a linear map such that
$\psi(e_0)={\bf 1}$, the relations (AL1-3) and (AL6) hold with
\begin{align*}
  h[z]= Y(\psi(h),z),\quad e_\al[z]= Y(\psi(e_\al),z)\quad
  \te{for }h\in\h,\  \al\in L,
\end{align*}
and such that the following relation holds for $\al,\be\in L$:
\begin{align}\label{A(L)457}
&(x-z)^{-\<\al,\be\>-1}e_{\alpha}[x]e_{\be}[z]
-(-z+x)^{-\<\al,\be\>-1}e_{\be}[z]e_{\alpha}[x]\nonumber\\
&\quad \quad \quad =  \varepsilon(\al,\be)e_{\al+\be}[z]x^{-1}\delta\left(\frac{z}{x}\right).
\end{align}
Then $\psi$ can be  extended uniquely to a nonlocal vertex algebra homomorphism from $V_L$ to $V$.
\end{prop}

\subsection{Vertex bialgebra $B_{L}$}

View $L$ as an abelian group and let $\C[L]$ be its group algebra with basis $\set{e^\al}{\al\in L}$.
Recall that $\h=\C\otimes_{\Z}L$, a vector space, and
$\wh\h^-=\h\otimes t^{-1}\C[t^{-1}]$, an abelian Lie algebra.
Note that both $\C[L]$ and $S(\wh\h^-)$ $(=U(\wh\h^-))$ are Hopf algebras.
Set
\begin{align}
  B_L= S(\wh\h^-)\ot \C[L],
\end{align}
which is a Hopf algebra. In particular, $B_{L}$ is a bialgebra,
where the comultiplication $\Delta$ and counit $\varepsilon$
are uniquely determined by
\begin{align}
  &\Delta(h(-n))=h(-n)\ot 1+1\ot h(-n),\quad \Delta(e^\al)=e^\al\ot e^\al,\\
  &\varepsilon(h(-n))=0,\quad \varepsilon(e^\al)=1
\end{align}
for $h\in\h,\  n\in \Z_{+},\ \al\in L$.
On the other hand, $B_L$ admits a derivation $\partial$ which is uniquely determined by
\begin{align}
  \partial(h(-n))=nh(-n-1),\quad \quad\partial(u\ot e^\al)= \al(-1)u\ot e^\al+\partial u \ot e^\al
\end{align}
for $h\in\h$, $n\in \Z_{+}$, $u\in S(\wh\h^-),\ \al\in L$.
Then $B_L$ becomes a commutative vertex algebra where the vertex operator map, denoted by $Y_{B_L}(\cdot,x)$,
is given by
$$Y_{B_L}(a,x)b=(e^{x\partial}a)b\   \   \   \mbox{ for }a,b\in B_{L}.$$
We sometimes denote this vertex algebra by $(B_{L},\partial)$. 

The following result was obtained in \cite{Li-smash}:

\begin{prop}\label{BL-vertex-operator}
For $h\in \h,\ \al\in L$, we have
\begin{align}
&Y_{B_{L}}(h(-1),x)=\sum_{n\ge 1}h(-n)x^{n-1},\\
  &Y_{B_L}(e^\al,x)=e^\al\exp\(\sum_{n\ge 1}\frac{\al(-n)}{n}x^n\)=e^\al E^{-}(-\al,x).
\end{align}
On the other hand,  $(B_L,\partial)$ is a vertex bialgebra
with the comultiplication $\Delta$ and counit $\varepsilon$ of $B_{L}$, which is
both commutative and cocommutative.
\end{prop}

Recall that the vertex algebra $V_{L}=S(\wh\h^-)\ot\C_{\epsilon}[L]$
contains $\h$ and $\C_{\epsilon}[L]$ as subspaces.
As the main result of this section,  we have:

\begin{thm}\label{lem:qva-rho-def}
There exists a right $B_L$-comodule vertex algebra structure $\rho$ on vertex algebra $V_L$,
which is uniquely determined by
\begin{align}
  \rho(h)=h\ot1+1\ot h(-1),\quad\rho(e_\al)=e_\al\ot e^\al\quad\te{for }h\in\h,\   \al\in L.
\end{align}
\end{thm}

\begin{proof} The uniqueness is clear as $V_{L}$ as a vertex algebra is generated by the subspace $\h+\C_{\epsilon}[L]$.
Denote by $Y^\ot$ the vertex operator map of the tensor product vertex algebra $V_L\ot B_L$.
Let $\rho^0$ be the linear map from $\h\oplus \C_{\epsilon}[L]$ to $V_L\ot B_L$, defined by
\begin{align*}
  \rho^0(h)= h\ot1+1\ot h(-1),\quad \rho^0(e_\al)=e_\al\ot e^\al \quad\te{for }h\in\h,\  \al\in L.
\end{align*}
We are going to apply Proposition \ref{prop:AQ-mod-hom-VA-hom} with $V=V_L\ot B_L$ and $\psi=\rho^0$.
As $B_{L}$ is a commutative vertex algebra, for $h,h'\in \h,\ \al\in L$ we have
\begin{align*}
&[Y^\ot(\rho^0(h),x),Y^\ot(\rho^0(h'),z)]=\<h,h'\>\frac{\partial}{\partial z} x^{-1}\delta\left(\frac{z}{x}\right),\\
&[Y^\ot(\rho^0(h),x),Y^\ot(\rho^0(e_{\al}),z)]=\<h,\al\>x^{-1}\delta\left(\frac{z}{x}\right)Y^\ot(\rho^0(e_{\al}),z).
\end{align*}
Let $\al,\be\in L$. Using Proposition \ref{BL-vertex-operator} we get
\begin{align*}
&Y^\ot(\rho^0(e_{\al}),x_1)Y^\ot(\rho^0(e_{\be}),x_2)\nonumber\\
=\  &Y(e_{\al},x_1)Y(e_{\be},x_2)\ot e^\al E^{-}(-\al,x_1)e^\be E^{-}(-\be,x_2)\nonumber\\
=\  &Y(e_{\al},x_1)Y(e_{\be},x_2)\ot e^{\al+\be} E^{-}(-\al,x_1)E^{-}(-\be,x_2)
\end{align*}
and
\begin{align*}
Y^\ot(\rho^0(e_{\be}),x_2)Y^\ot(\rho^0(e_{\al}),x_1)
=Y(e_{\be},x_2)Y(e_{\al},x_1)\ot e^{\al+\be} E^{-}(-\be,x_2)E^{-}(-\al,x_1).
\end{align*}
Note that $E^{-}(-\al,x_1)E^{-}(-\be,x_2)=E^{-}(-\be,x_2)E^{-}(-\al,x_1)$.
Furthermore, we have 
\begin{align*}
&(x_1-x_2)^{-\<\al,\be\>-1}Y^\ot(\rho^0(e_{\al}),x_1)Y^\ot(\rho^0(e_{\be}),x_2)\nonumber\\
&\   \   \   \    -(-x_2+x_1)^{-\<\al,\be\>-1}Y^\ot(\rho^0(e_{\be}),x_2)Y^\ot(\rho^0(e_{\al}),x_1)\nonumber\\
=\  &\epsilon(\al,\be)Y(e_{\al+\beta},x_2)x_1^{-1}\delta\left(\frac{x_2}{x_1}\right)\ot e^{\al+\be}E^{-}(-\al,x_1)E^{-}(-\be,x_2)
\nonumber\\
=\ &\epsilon(\al,\be)Y(e_{\al+\beta},x_2)\ot e^{\al+\be}E^{-}(-\al,x_2)E^{-}(-\be,x_2)x_1^{-1}\delta\left(\frac{x_2}{x_1}\right)
\nonumber\\
=\ &\epsilon(\al,\be)Y(e_{\al+\beta},x_2)\ot e^{\al+\be}E^{-}(-\al-\be,x_2)x_1^{-1}\delta\left(\frac{x_2}{x_1}\right)
\nonumber\\
=\ &\epsilon(\al,\be)Y^{\ot}(e_{\al+\beta},x_2)x_1^{-1}\delta\left(\frac{x_2}{x_1}\right)
\end{align*}
and
\begin{align*}
&\frac{d}{dx}Y^{\ot}(e_{\al},x)
=\ \frac{d}{dx}Y(e_{\al},x)\ot e^{\al}E^{-}(-\al,x)+Y(e_{\al},x)\ot e^{\al}\al(x)^{+}E^{-}(-\al,x)\nonumber\\
=\  &\(\al(x)^{+}Y(e_{\al},x)+Y(e_{\al},x)\al(x)^{-}\)
\ot Y(e^{\al},x)+Y(e_{\al},x)\ot \al(x)^{+}Y(e^{\al},x)\nonumber\\
=\ &Y^{\ot}(\rho^0(\al),x)^+Y^{\ot}(\rho^0(e_{\al}),x)+Y^{\ot}(\rho^0(e_{\al}),x)Y^{\ot}(\rho^0(\al),x)^-,
\end{align*}
where $Y^{\ot}(\rho^0(\al),x)^+=\sum_{n< 0}\rho^0(\al)_{n}x^{-n-1}$
and $Y^{\ot}(\rho^0(\al),x)^-=\sum_{n\ge 0}\rho^0(\al)_{n}x^{-n-1}$.
(Notice that $Y_{B_{L}}(h,x)^{-}=0$ for $h\in \h$.)
Then by Proposition \ref{prop:AQ-mod-hom-VA-hom}
there exists a vertex algebra homomorphism $\rho$ from $V_L$ to $V_L\ot B_L$,
which extends $\rho^0$ uniquely.

Recall that $\Delta: B_{L}\rightarrow B_{L}\otimes B_{L}$ is a homomorphism of vertex algebras.
Notice that  both $(\rho\ot1)\rho$ and $(1\ot\Delta)\rho$ are vertex algebra homomorphisms
from $V_{L}$ to $V_{L}\otimes B_{L}\otimes B_{L}$. It can be readily  seen that
$(\rho\ot1)\rho$ and $(1\ot\Delta)\rho$ agree on the subspace $\h+ \C_{\epsilon}[L]$ of $V_{L}$.
Since $V_{L}$ is generated by $\h+ \C_{\epsilon}[L]$, we conclude
\begin{align*}
  (\rho\ot 1)\rho=(1\ot\Delta)\rho
\end{align*}
(on $V_{L}$). Similarly, we can prove
\begin{align*}
  (1\ot \varepsilon)\rho(u)=u\quad\te{for }u\in V_{L}.
\end{align*}
Therefore, $\rho$ is a right $B_L$-comodule vertex algebra structure on $V_L$.
\end{proof}

Recall the $\epsilon$-twisted group algebra $\C_{\epsilon}[L]$. Set
\begin{eqnarray}
B_{L,\epsilon}=S(\wh\h^-)\ot\C_{\epsilon}[L],
\end{eqnarray}
viewed as an associative algebra.
From \cite{Li-smash}, $B_{L,\epsilon}$ admits a derivation $\partial$ such that
$$\partial e_{\alpha}=\alpha(-1)\ot e_{\alpha},\   \   \   \partial h(-n)=nh(-n-1)$$
for $\alpha\in L,\ h\in \h,\ n\ge 1$.
Then we have a  nonlocal vertex algebra $(B_{L,\epsilon},\partial)$, where
$$Y(e_\al,x)=E^{-}(-\al,x)e_\al\   \   \   \   \mbox{ for } \al\in L.$$
Furthermore, it was proved in \cite{Li-smash} that $B_{L,\epsilon}$ is a $B_L$-module nonlocal vertex algebra
where the $B_{L}$-module structure $Y_{B_{L,\epsilon}}^{B_L}(\cdot,x)$ on $B_{L,\epsilon}$
is uniquely determined by
\begin{eqnarray}
Y_{B_{L,\epsilon}}^{B_L}(e^{\alpha},x)=E^{+}(-\alpha,x)x^{\alpha(0)}\   \   \   \mbox{ for }\alpha\in L.
\end{eqnarray}
It was proved therein that the comultiplication $\Delta: B_L\rightarrow B_L\ot B_L$
 gives rise to a vertex algebra homomorphism
from $V_{L}$ into $B_{L,\epsilon}\sharp B_L$, with $V_L$ and $B_{L,\epsilon}$ being canonically identified
with $B_L$.

The following is an analogue of Theorem \ref{lem:qva-rho-def}:

\begin{prop}\label{BL-epsilon-comodule}
There exists a nonlocal vertex algebra homomorphism
 $\rho: B_{L,\epsilon}\rightarrow B_{L,\epsilon}\otimes B_L$, which is uniquely determined by
\begin{eqnarray}
  \rho(h(-n))=h(-n)\ot1+1\ot h(-n),\quad \rho(e_{\alpha})=e_{\alpha}\otimes e^{\alpha}
\end{eqnarray}
for $h\in \h,\ n\in \Z_+,\ \alpha\in L$. Furthermore, $B_{L,\epsilon}$ with $\rho$ becomes
 a right $B_{L}$-comodule nonlocal vertex algebra and $\rho$ is compatible with $Y_{B_{L,\varepsilon}}^{B_L}(\cdot,x)$.
\end{prop}

\begin{proof} Let $\Delta$ be the comultiplication map of $S(\wh\h^-)$, which is an algebra homomorphism.
Define a linear map $\rho:  B_{L,\epsilon}\rightarrow B_{L,\epsilon}\otimes B_L$ by
$$\rho( u\ot e_\al)=\Delta(u)(e_\al\ot e^\al)\    \    \    \mbox{ for }u\in S(\wh\h^-),\ \al\in L. $$
 As $B_{L}$ is commutative, it can be readily seen that
  the linear map from $\C_{\varepsilon}[L]$ to $\C_{\varepsilon}[L]\ot \C[L]$,
sending $e_\al$ to $e_\al\ot e^\al$ for $\al\in L$, is an algebra homomorphism.
Consequently, $\rho$ is an algebra homomorphism. It is straightforward to see that
$$\rho(\partial h(-n))=(\partial \ot 1+1\ot \partial)\rho (h(-n)),\   \   \   \
\rho (\partial e_\al)=(\partial \ot 1+1\ot \partial)\rho (e_\al)$$
for $h\in \h,\ n\in \Z_+,\  \al\in L$.
Since $B_{L,\varepsilon}$ as an algebra is generated by $\wh \h^{-}+\C_{\varepsilon}[L]$,
it follows that $\rho$ is a homomorphism of differential algebras.
Thus $\rho$ is a  nonlocal vertex algebra homomorphism.
The same argument at the end of the proof of Theorem  \ref{lem:qva-rho-def} shows
that $B_{L,\epsilon}$ with $\rho$ is a right $B_{L}$-comodule. Therefore, $B_{L,\epsilon}$ with $\rho$ becomes
 a right $B_{L}$-comodule nonlocal vertex algebra.

Let $\al,\be\in L$.  As $Y_{B_{L,\varepsilon}}^{B_L}(e^\al,x)e_\be=E^{+}(-\al,x)x^{\al(0)}e_\be=x^{\<\al,\be\>}e_\be$,
we have
$$\rho(Y_{B_{L,\varepsilon}}^{B_L}(e^\al,x)e_\be)=\rho\(x^{\<\al,\be\>}e_{\be}\)=
x^{\<\al,\be\>}(e_{\be}\ot e^{\be})=(Y_{B_{L,\varepsilon}}^{B_L}(e^\al,x)\ot 1)\rho(e_\be).$$
On the other hand, we have
\begin{align*}
  Y_{B_{L,\varepsilon}}^{B_{L}}(e^\al,x)e_{\be}\in \C_{\epsilon}[L]\ot\C((x)).
\end{align*}
Since $B_L$  as a vertex algebra is generated by $\C[L]$
and $B_{L,\varepsilon}$  as a nonlocal vertex algebra is generated by $\C_{\epsilon}[L]$,
from Lemma \ref{lem:Y-M-rho-compatible}
 $Y_{B_{L,\varepsilon}}^{B_L}(\cdot,x)$ is compatible with $\rho$.
\end{proof}

\begin{rem}
{\em  With Proposition \ref{BL-epsilon-comodule}, by Theorem \ref{prop:deform-va}
we get a new nonlocal vertex algebra structure on $B_{L,\epsilon}$.
It was essentially proved in \cite{Li-smash} that vertex algebra $V_{L}$ coincides
with (is canonically isomorphic to) this deformation of $B_{L,\epsilon}$. }
\end{rem}

Recall that for $a\in \h,\ f(x)\in \C((x))$,
$$\Phi(a,f)(z)=\sum_{n\ge 0}\frac{(-1)^{n}}{n!}f^{(n)}(z)a_{n}$$
on $V_{L}$. We have
\begin{align}\label{formula-Phi}
\Phi(a,f)(z)e_{\be}=\<\be,a\>f(z)e_{\be}\   \   \   \mbox{ for }\be\in L.
\end{align}
Now we define $\Phi(G(x))(z)$ for $G(x)\in \h\otimes \C((x))$ by linearity.
On the other hand, we extend the bilinear form $\<\cdot,\cdot\>$ on $\h$ to
a $\C((x))$-valued bilinear form on $\h\otimes \C((x))$. Then we have
$$\Phi(G(x))(z)e_{\be}=\<\be,G(z)\>e_{\be}$$
for $G(x)\in \h\otimes \C((x)),\ \be\in L$.

Note that for any $g(x)\in x\C[[x]]$, $e^{g(x)}$ exists in $\C[[x]]$.
Recall from Theorem \ref{lem:qva-rho-def}
the right $B_L$-comodule vertex algebra structure $\rho: V_{L}\rightarrow V_{L}\otimes B_{L}$  on $V_L$.

\begin{prop}\label{lem:qva-Y-M-def}
Let $f: \h\rightarrow \h\otimes x\C[[x]];\ \al\mapsto f(\al,x)$ be any linear map.
Then there exists a  $B_L$-module structure $Y_M^f(\cdot,x)$ on $V_{L}$,
which is uniquely determined by
\begin{eqnarray}
 Y_M^f(e^\al,z) =\exp (\Phi( f(\alpha,x))(z))\   \   \   \mbox{ for } \al\in L,
\end{eqnarray}
 and $V_{L}$ with this $B_L$-module structure becomes a $B_L$-module vertex algebra.
 Furthermore,  $Y_M^f(\cdot,x)$ is compatible with $\rho$ and $Y_M^f$ is an invertible element
 of $\mathfrak L_{B_L}^\rho(V_L)$.
\end{prop}

\begin{proof}  Since $V_L$  as a $V_{L}$-module is irreducible,
by \cite[Corollary 3.10]{li-qva2} $V_L$ is non-degenerate.
 By Proposition \ref{prop:B-V-vertex-bialg},  $B(V_L)$
is a vertex bialgebra and $V_L$ is naturally a $B(V_L)$-module vertex algebra.

For $\al\in \h$, as $f(\al,x)\in \h\ot x\C[[x]]$, $\exp(\Phi( f(\alpha,x))(z))$
 exists and from \cite[Proposition 2.10]{Li-smash}, we have
\begin{align*}
\exp(\Phi( f(\alpha,x))(z)) \in \PEnd(V_L)\subset B(V_L),
\end{align*}
a group-like element, i.e., $\Delta(\exp(\Phi( f(\alpha,x))))=\exp(\Phi( f(\alpha,x)))\ot \exp(\Phi( f(\alpha,x)))$.
Since $\Phi$ is linear and $f$ is linear in $\alpha\in \h$, we have
\begin{align*}
 \exp(\Phi(f(\alpha,x)))  \exp\(\Phi\( f(\beta,x) \)\)
 =\exp(\Phi( f(\alpha+\beta,x))) \quad\te{for }\al,\be\in L.
\end{align*}
Then the linear map  $\psi:\C[L]\rightarrow B(V_{L})$
defined by $\psi(e^\al)=\exp(\Phi( f(\alpha,x)))$ for $\al\in L$ is an algebra homomorphism.
Note that though $B(V_L)$  is not necessarily commutative,
 the  subalgebra  generated by
$$\Phi(f^{(r)}(\al,x)),\   \exp(\Phi(f(\al,x)))\   \   \mbox{ for }r\ge 0,\ \al\in\h$$
is a commutative differential subalgebra of $B(V_L)$.
Then by \cite[Lemma 5.4]{Li-smash}, $\psi$ can be extended uniquely
to a homomorphism $\bar{\psi}$ of differential algebras from $B_L$ to $B(V_{L})$.
Furthermore,  $\bar{\psi}$ is a homomorphism of vertex algebras.
We have
\begin{eqnarray*}
&&(\bar{\psi}\otimes \bar{\psi})\Delta_{B_L}(e^\al)= \exp(\Phi(f(\alpha,x)))\ot  \exp(\Phi(f(\alpha,x)))
=  \Delta_{B}\bar{\psi} (e^\al),\nonumber\\
&& \varepsilon_{B}(\bar{\psi}(e^\al))= \varepsilon_{B}(\exp(\Phi(f(\alpha,x))))=1=\varepsilon_{B_L} (e^\al)
\end{eqnarray*}
for $\al\in L$, where $(\Delta_{B_L}, \varepsilon_{B_L})$ and $(\Delta_{B},\varepsilon_{B})$ denote
the coalgebra structures of $B_{L}$ and $B(V_{L})$, respectively.
Since $\C[L]$ generates $B_{L}$ as a differential algebra,
it follows that  $\bar{\psi}$ is a homomorphism of differential bialgebras.
Consequently, $\bar{\psi}$ is a homomorphism of vertex bialgebras.
Then,  by \cite[Proposition 4.8]{Li-smash}, $\bar{\psi}$ gives rise to
 a $B_L$-module structure $Y_M^f(\cdot,x)$ on $V_L$,
which makes $V_L$ a $B_L$-module vertex algebra.

Next, we prove that $Y_M^f(\cdot,x)$ is compatible with $\rho$, i.e.,
\begin{align}\label{eq:half-action}
  \rho(Y_M^f(b,x)v)=\(Y_M^f(b,x)\ot 1\)\rho(v)\   \   \mbox{ for }b\in B_L,\  v\in V_L.
\end{align}
Let $\al,\be\in L$. By (\ref{formula-Phi}) we have
$\Phi(f(\al,x))e_{\beta}=\<\be,f(\al,x)\>e_{\beta}$, so that
\begin{eqnarray*}
Y_M^f(e^\al,x)e_\be=
  \exp( \Phi(f(\al,x))) e_\be=e^{\<\be,f(\al,x)\>}e_{\be}.
\end{eqnarray*}
Then
$$\rho(Y_M^f(e^\al,x)e_\be)=\rho\(e^{\<\be,f(\al,x)\>}e_{\be}\)=
e^{\<\be,f(\al,x)\>}(e_{\be}\ot e^{\be})=(Y_M^f(e^\al,x)\ot 1)\rho(e_\be).$$
This shows that \eqref{eq:half-action} holds for $b=e^\al,\ v=e_{\be} \in L$.
On the other hand, we have
\begin{align*}
  Y_M^f(e^\al,x)e_{\be}\in \C_{\epsilon}[L]\ot\C((x)).
\end{align*}
Since $B_L$  as a vertex algebra is generated by $\C[L]$,
and $V_L$  as a vertex algebra is generated by $\C_{\epsilon}[L]$,
it follows from Lemma \ref{lem:Y-M-rho-compatible} that \eqref{eq:half-action} holds
for all $u\in B_L,\  v\in V_L$. Thus $Y_M^f(\cdot,x)$ is compatible with $\rho$.
Namely, $Y_M^f\in\mathfrak L_{B_L}^\rho(V_L)$.

Note that with $-f$ in place of $f$ we get another $B_L$-module vertex algebra structure $Y_M^{-f}$
on $V_L$ such  that $Y_M^{-f}\in\mathfrak L_{B_L}^\rho(V_L)$, where
\begin{align}\label{eq:qva-Y-M-rho-compatible-invertible-temp}
Y_M^{-f}(e^\al,x)=\exp(\Phi(-f(\al,x))=\exp(-\Phi(f(\al,x))
\end{align}
for $\al\in L$.
It follows that $(Y_M^f\ast Y_M^{-f})(e^\al)=1=\varepsilon(e^\al)$.
Again, since $B_L$ as a vertex algebra is generated by $\C[L]$,
we have $Y_M^f\ast Y_M^{-f}=Y_{M}^\varepsilon$.
That is, $Y_M^{-f}$ is an inverse of $Y_M^f$ in $\mathfrak L_{B_L}^\rho(V_L)$.
\end{proof}

\begin{rem}\label{phi-compatibility-VL}
{\em Let $\phi(x,z)$ be an associate given as before.
Note that if the linear map $f: \h\rightarrow \h\otimes x\C[[x]]$ in Proposition \ref{lem:qva-Y-M-def}
is assumed to be a linear map $f: \h\rightarrow \h\otimes (x\C[[x]]\cap \C_{\phi}((x)))$, then
using Lemma \ref{phi-compatible-va-tech} (a technical result) we can show that the $B_{L}$-module structure $Y_M^f(\cdot,x)$ on $V_{L}$
 is also $\phi$-compatible.}
\end{rem}

Combining Theorems \ref{prop:deform-va} and \ref{thm:S-op}, we have:

\begin{thm}\label{thm:qlva}
Let $f: \h\rightarrow \h\otimes x\C[[x]]$ be any linear map.
Then there exists a quantum vertex algebra structure on $V_L$,
whose vertex operator map, denoted by $Y^f(\cdot,x)$,
 is uniquely determined by
\begin{eqnarray}
&&Y^{f}(e_{\alpha},x)=Y(e_\al,x)\exp(\Phi(f(\alpha,x)))\quad \te{for }\al\in L, \label{e-alpha-expression} \\
&&Y^{f}(h,x)=Y(h,x)+\Phi(f'(h,x))\quad \te{for }h\in \h.\label{h-expression}
\end{eqnarray}
Denote this quantum vertex algebra by $V_{L}^f$.
Then $V_L^f$ is an irreducible $V_L^f$-module and in particular, $V_L^f$ is non-degenerate.
Furthermore,
the following relations hold:
\begin{align}
&[Y^{f}(a,x),Y^{f}(b,z)]\nonumber\\
=&\<a,b\>\frac{\partial}{\partial z}x^{-1}\delta\left(\frac{z}{x}\right)
-\<f''(a,x-z),b\>+\<f''(b,z-x),a\>,\label{new-h-h-relations}\\
&[Y^f(a,x),Y^f(e_\be,z)]=\<a,\be\>Y^f(e_\be,z)x^{-1}\delta\(\frac{z}{x}\)\nonumber\\
-&Y^f(e_\be,z)(\<f'(\be,z-x),a\>+\<f'(a,x-z),\be\>),
\label{new-h-e-relations}\\
&\te{Res}_xx^{-1}Y^f(\al,x)e_\al=\mathcal D e_\al+\<f'(\al,0),\al\>e_\al,\\
&(x-z)^{-\<\al,\be\>-1}Y^{f}(e_{\alpha},x)Y^{f}(e_{\be},z)\nonumber\\
&\   \   \   \   -(-z+x)^{-\<\al,\be\>-1}
e^{\<\beta,f(\al,x-z)\>-\<\al,f(\be,z-x)\>}Y^{f}(e_{\be},z)Y^{f}(e_{\alpha},x)
\nonumber\\
=\  &\varepsilon(\al,\be)Y^f(e_{\al+\be},z)x^{-1}\delta\left(\frac{z}{x}\right)
\label{relation-eal-ebe}
\end{align}
for $a,b\in \h,\  \al,\be\in L$.
\end{thm}

\begin{proof}  For the first assertion, the uniqueness is clear, so we only need to show the existence.
 With Propositions \ref{lem:qva-rho-def} and \ref{lem:qva-Y-M-def},
by Theorems \ref{prop:deform-va} and \ref{thm:S-op} we have a quantum vertex algebra
$\mathfrak D_{Y_M^f}^\rho (V_L)$ with $V_L$ as its underlying space.
From Theorem \ref{prop:deform-va}, for $h\in \h,\ \alpha\in L$, we have
\begin{eqnarray*}
Y^{f}(h,x)&=&Y^{\sharp}(\rho(h),x)=Y^{\sharp}(h\ot {\bf 1}+{\bf 1}\ot h,x)
=Y(h,x)+Y_{M}^f(h,x)\nonumber\\
&=&Y(h,x)+\Phi(f'(h,x)),\\
Y^{f}(e_{\alpha},x)&=&Y^{\sharp}(\rho(e_\al),x)=Y^{\sharp}(e_\al\ot e^\al,x)=Y(e_\al,x)Y_{M}^f(e^\al,x) \nonumber\\
&=&Y(e_\al,x)\exp(\Phi(f(\alpha,x))).
\end{eqnarray*}
Thus we have a quantum vertex algebra structure on $V_L$ as desired.

For the second assertion, since $V_L$ is an irreducible $V_L$-module, it suffices to prove that
  any $V_L^f$-submodule $U$ of $V_L^f$ is also a $V_L$-submodule of $V_L$.
For $h\in \h$, since $f'(h,x)\in \h\ot x\C[[x]]$, we have
\begin{align}
&Y^{f}(h,x)^{-}=Y(h,x)^{-}=\sum_{n\ge 0}h(n)x^{-n-1},\label{Yf-h-m}\\
&Y^{f}(h,x)^{+}=Y(h,x)^{+}+\Phi(f'(h,x)).\label{Yf-h-p}
\end{align}
From (\ref{Yf-h-m}) we conclude $h(n)U\subset U$ for $h\in \h,\ n\ge 0$. Furthermore, we get
$\Phi(f'(h,x))U\subset U[[x]]$ for $h\in \h$. Then from (\ref{Yf-h-p}) we deduce $Y(h,x)^{+}U\subset U[[x]]$.
This proves that $U$ is an $\wh{\h}$-submodule of $V_L$. Furthermore, for $\al\in L$, as
$Y(e_\alpha,x)=Y^f(e_\alpha,x)\exp(-\Phi(f(\alpha,x)))$ we obtain
$Y(e_\alpha,x)U\subset U[[x,x^{-1}]]$.
Then it follows that $U$ is also a $V_L$-submodule of $V_L$.
Therefore, $V_L^f$ is an irreducible $V_L^f$-module, which implies that
 $V_L^f$ is a simple quantum vertex algebra and it is non-degenerate by \cite{li-qva2}.

For the furthermore assertion, note that for $u,v\in \h,\ g(x)\in \C((x))$, we have
\begin{align}
  \Phi(u\ot g(x))v=\sum_{n\ge 0}(-1)^n\frac{1}{n!}g^{(n)}(x)u_nv
  =-\<u\ot g'(x),v\>.
\end{align}
Combining this with \eqref{formula-Phi} and Theorem \ref{thm:S-op}, we get that
\begin{eqnarray}
&&{S}(x)(b\otimes a)=b\ot a+{\bf 1}\ot {\bf 1}\ot (\<f''(b,x),a\>-\<f''(a,-x),b\>),\label{eq:S-h-h-rel}\\
&&S(x)(b\ot e_{\al})=b\ot e_\al-\vac\ot e_\al\ot (\<f'(\al,-x),b\>+\<f'(b,x),\al\>),\label{eq:S-h-e-rel}\\
&&S(x)(e_\al\ot b)=e_\al\ot b+ e_\al\ot \vac\ot (\<f'(\al,x),b\>+\<f'(b,-x),\al\>),\\
&&{S}(x)(e_{\beta}\otimes e_{\al})=(e_{\be}\otimes e_{\al})\ot e^{\<\beta,f(\al,-x)\>-\<\al,f(\be,x)\>},\label{eq:S-e-e-rel}
\end{eqnarray}
where $a,b\in \h$ and $\al,\be\in L$ and $S(x)$ is the quantum Yang-Baxter operator of $V_L^f$.
Since $f(\h,x)\subset \h\ot x\C[[x]]$, we have that
\begin{align*}
  \Phi(f(a,x))e_\be\in \C e_\be\ot x\C[[x]]\quad\te{and}\quad
  \Phi(f(a,x))b\in \C\vac\ot x\C[[x]].
\end{align*}
It follows that
\begin{align}
  \te{Sing}_x\Phi(f'(a,x))b=0=\te{Sing}_x\Phi(f'(a,x))\be=\te{Sing}_x\exp\(\Phi(f(a,x))\)b
\end{align}
and
\begin{align*}
  \exp\(\Phi(f(a,x))\)e_\be \in e_\be +\C e_\be\ot x\C[[x]].
\end{align*}
Then we get that
\begin{align*}
  &\te{Sing}_x Y^f(a,x)b=\te{Sing}_x Y(a,x)b,\quad
  \te{Sing}_x Y^f(a,x)e_\be=\te{Sing}_x Y(a,x)e_\be,\\
  &\te{Res}_xx^{-1} Y^f(\al,x)e_\al=\partial e_\al+\<f'(\al,0),\al\>e_\al,\\
  &\te{Sing}_x x^{\<\al,\be\>}Y^f(e_\al,x)e_\be=\varepsilon(\al,\be)e_{\al+\be}.
\end{align*}
Combining these with Lemma \ref{nonlocalva-fact}, Definition \ref{de:qva} and relations \eqref{eq:S-h-h-rel}, \eqref{eq:S-h-e-rel}, \eqref{eq:S-e-e-rel}, we complete the proof.
\end{proof}

Combining Proposition \ref{prop:AQ-mod-hom-VA-hom} and Theorem \ref{thm:qlva}, we immediately get that
\begin{coro}
Let $f:\h\to \h\ot x\C[[x]]$ be a linear map such that
\begin{align*}
  \<f(a,x),b\>=\<f(b,-x),a\>,\quad \te{for }a,b\in\h.
\end{align*}
Then the identity map on $\h\oplus\C_\varepsilon[L]$ determines a nonlocal vertex algebra isomorphism from $V_L$ to $V_L^f$.
\end{coro}


Recall that $\h=\C\otimes_\Z L\subset V_L$.
Let $G$ be an automorphism group of $V_{L}$ with a linear character $\chi$ such that $G$ preserves $\h$.
It follows that $G$ preserves the bilinear form on $\h$ and the standard conformal vector
$\omega$ of $V_L$. Thus $gL(0)=L(0)g$ on $V_L$ for $g\in G$.
Then $V_L$ is a $(G,\chi)$-module vertex algebra with $R(g)=\chi(g)^{L(0)}g$ for $g\in G$.
Furthermore, we have:

\begin{prop}\label{VL-eta-G-algebra}
Let $G$ be an automorphism group of $V_{L}$  such that $G(\h)=\h$ and let $\chi$ be a linear character of $G$.
Assume  $\eta(\cdot,x): \h \rightarrow \h\ot x\C[[x]]$ is a linear map satisfying the condition
\begin{eqnarray}\label{eta-G-equivarance}
(\mu\ot 1) \eta(h,x)=\eta(\mu(h),\chi(\mu)x) \   \   \   \mbox{ for }\mu\in G,\ h\in \h.
\end{eqnarray}
Then $V_L^\eta$ is a $(G,\chi)$-module quantum vertex algebra with $R$ defined by
$R(g)=\chi(g)^{L(0)}g$ for $g\in G$.
\end{prop}

\begin{proof} Recall that for $v\in V_L,\ f(t)\in \C((t))$,
$$\Phi(v\ot f(t))(z)=\text{Res}_{x_1}Y(v,x_1)f(z-x_1).$$
 For $\mu\in G$, set $R_{\mu}=\chi(\mu)^{L(0)}\mu$.
For $h\in \h$, we have
\begin{align}\label{R-mu-e-alpha}
R_{\mu}\Phi(\eta(h,t))(z)=\  &\text{Res}_{x_1}R_{\mu} Y(\eta(h,z-x_1),x_1)\nonumber\\
=\  &\text{Res}_{x_1}\chi(\mu)Y((\mu\ot 1)\eta(h,z-x_1),\chi(\mu)x_1)R_{\mu}\mu\nonumber\\
=\  &\text{Res}_{x_1}\chi(\mu)Y(\eta(\mu(h),\chi(\mu)(z-x_1)),\chi(\mu)x_1)R_{\mu}\nonumber\\
=\  &\text{Res}_{z_1}Y(\eta(\mu(h),\chi(\mu)z-z_1),z_1)R_{\mu}\nonumber\\
=\  &\Phi(\eta(\mu(h),t))(\chi(\mu)z)R_{\mu}.
\end{align}
Then for $\al\in L$ we get
\begin{align*}
R_{\mu} Y^\eta(e_\al,z)R_\mu^{-1}
 =\ &R_{\mu} Y(e_\al,z)\exp(\Phi(\eta(\al,t))(z))R_\mu^{-1}\\
=\  &Y(R_{\mu}e_\al,\chi(\mu)z)\exp\(R_{\mu}\Phi(\eta(\al,t))(z)R_\mu^{-1}\)\\
=\  &\chi(\mu)^{\frac{1}{2}\<\al,\al\>}Y(\mu e_\al,\chi(\mu)z)\exp\(\Phi(\eta(\mu(\al),t)(\chi(\mu)z)\)\\
=\  &\lambda_{\mu,\al}\chi(\mu)^{\frac{1}{2}\<\mu(\al),\mu(\al)\>}
Y(e_{\mu (\al)},\chi(\mu)z)\exp\(\Phi(\eta(\mu(\al),t))(\chi(\mu)z)\)\\
=\ &\lambda_{\mu,\al}\chi(\mu)^{\frac{1}{2}\<\mu(\al),\mu(\al)\>}
Y^{\eta}(e_{\mu (\al)},\chi(\mu)z)\\
=\ &Y^{\eta}(R_{\mu}e_{\al},\chi(\mu)z),
\end{align*}
where $\mu(e_\al)=\lambda_{\mu,\alpha}e_{\mu (\al)}$ with $\lambda_{\mu,\al}\in \C$.
Note that (\ref{eta-G-equivarance}) implies
$$(\mu\ot 1) \eta'(h,x)=\chi(\mu)\eta'(\mu(h),\chi(\mu)x). $$
Following the argument in (\ref{R-mu-e-alpha}) we obtain
\begin{eqnarray*}
R_{\mu}\Phi(\eta'(h,t))(z)R_{\mu}^{-1}=\chi(\mu)\Phi(\eta'(\mu(h),t))(\chi(\mu)z)
=\Phi(\eta'(R_{\mu}(h),t))(\chi(\mu)z).
\end{eqnarray*}
Using this we get
$$R_{\mu}Y^\eta(h,x)R_{\mu}^{-1}=R_{\mu}\(Y(h,x)+\Phi(\eta'(h,t))(x)\)R_{\mu}^{-1}=Y^\eta(R_\mu h,\chi(\mu)x).$$
Note that $\{ e_{\alpha}\ |\  \al\in L\}\cup \h$ is a generating subset  of $V_L^\eta$ as a nonlocal vertex algebra.
By using weak associativity (or by Lemma \ref{lem:G-va-generating}) it is straightforward  to show that
$$R_{\mu} Y^\eta(v,x)R_\mu^{-1}=Y^{\eta}(R_{\mu}v,\chi(\mu)x)\   \   \   \mbox{ for all }\mu\in G,\ v\in V_L.$$
Therefore, $V_{L}^{\eta}$ is a $(G,\chi)$-module quantum vertex algebra as desired.
\end{proof}

\begin{rem}\label{construction-equivariant-function-algebra}
{\em We here show  the existence of a linear map $\eta(\cdot,x)$ satisfying the condition (\ref{eta-G-equivarance})
in Proposition \ref{VL-eta-G-algebra} with $G$ a finite group.
Let $g(\cdot,x):\  \h\rightarrow \h\ot x\C[[x]]$ be any linear map.
Define a linear map $\widetilde{g}(\cdot,x):\  \h\rightarrow \h\ot x\C[[x]]$ by
\begin{align}
\widetilde{g}(\alpha,x)=\sum_{\sigma\in G} (\sigma\ot 1)g(\sigma^{-1}\al,\chi(\sigma^{-1})x)
\   \   \  \mbox{ for }\alpha\in \h.
\end{align}
Then it is straightforward to show
\begin{align}
(\sigma\ot 1)\widetilde{g}(\alpha,x)=\widetilde{g}(\sigma \al,\chi(\sigma)x)\quad \te{ for }\sigma\in G,\ \alpha\in \h.
\end{align}}
\end{rem}

\end{document}